\documentclass[12pt,leqno]{amsart}
\usepackage{times, color}
\usepackage{amsmath,amssymb,amscd,amsthm}
\usepackage{latexsym}
\usepackage{epsfig}
\usepackage[shortlabels]{enumitem}

\newtheorem{theorem}{Theorem}
\newtheorem{lemma}{Lemma}
\newtheorem{proposition}{Proposition}
\newtheorem{remark}{Remark}

\newtheorem{definition}{Definition}
\newtheorem{corollary}{Corollary}

\newcommand{\pr}[1]{\left(#1 \right)}

\newcommand{\n}[2]{{\left\| #1 \right\|}_{#2}}
\newcommand{\f}[2]{\frac{#1}{#2}}
\newcommand{\lan}[1]{\left\langle #1\right\rangle}
\newcommand{\wh}[1]{\widehat{#1}}
\newcommand{\tn}[1]{\textnormal{#1}}
\newcommand{\wt}[1]{\widetilde{#1}}

\newcommand{\ga}{\gamma}
\newcommand{\Ga}{\Gamma}
\newcommand{\de}{\delta}

\newcommand{\ve}{\varepsilon}

\newcommand{\si}{\sigma}

\newcommand{\Om}{\Omega}

\newcommand{\br}{{\mathbb R}}

\newcommand{\bt}{{\mathbb T}}
\newcommand{\bz}{\mathbb Z}
\newcommand{\bp}{\mathbf P}
\newcommand{\bn}{\mathbb{N}}

\newcommand{\nr}{\mathcal{NR}}

\newcommand{\bb}{\mathcal{HL}}
\newcommand{\cf}{\mathcal F}

\newcommand{\rr}{\mathcal R}

\newcommand{\p}{\partial}

\newcommand{\ds}{\displaystyle}

\title[Smoothing and growth bound of polynomial-gKdV]{Smoothing and growth bound of periodic generalized Korteweg-de Vries equation}
\author{Seungly Oh, Atanas G. Stefanov}
\begin{document}
	
	\thanks{Atanas Stefanov is partially supported by NSF-DMS 1908626.}
\maketitle

\begin{abstract}
For generalized KdV models with polynomial nonlinearity, we establish a local smoothing property in $H^s$ for $s>\f{1}{2}$.  Such smoothing effect persists  globally, provided that the $H^1$ norm does not blow up in finite time. More specifically, we show that a translate of the nonlinear part of the solution gains $\min(2s-1,1)-$ derivatives for $s>\f{1}{2}$.  Following a new simple method, which is of independent interest, we establish that, for $s>1$, $H^s$~norm of a solution grows at most by $\lan{t}^{s-1+}$ if $H^1$~norm is a priori controlled.
\end{abstract}

In this article, we shall be concerned with real-valued solutions of the periodic generalized KdV given by
\begin{equation}\label{eq:gKdV}
\left|\begin{array}{l}u_t + u_{xxx} =  \p_x(P(u)), \qquad x \in \bt, \\
u(0) = f \in H^s,  \qquad s> \f{1}{2}\end{array}\right.
\end{equation}
where $P(\cdot)$ is a polynomial.  Let us take the opportunity to review some recent developments about the model \eqref{eq:gKdV} as well as other related dispersive models, as they pertain to our  goal in this article. 
\subsection{Cauchy theory, local smoothing effect and polynomial bounds}
Classical models which are given by  $P(u)=u^2, u^3$ are the ubiquitous KdV and mKdV models which are completely integrable.  These are easily the most well-studied PDE models, modeling a uni-directional propagation of surface water wave along shallow channels.  

Sharp results for Cauchy problems involving these models are available.   Using inverse scattering, KdV (resp. mKdV) is globally well-posed on $H^{-1}$ (resp. $L^2$) \cite{KP1, KP2} which is shown to be sharp in \cite{mol1, mol2}.  Using perturbative approaches, KdV (resp. mKdV) on the torus is known to be globally well-posed in $H^s$ for $s\geq -\f{1}{2}$ (resp. $s\geq \f{1}{3}$) \cite{Bour, I1, KPV, japanese}.  Also see \cite{titi, KO, MDV, Zhou} for unconditional uniqueness results for these models.  Some authors have also considered models with a mixed nonlinearity, in the form $P(u)=a u^2+b u^3$ - referred to as the Gardner model. Gardner equation can be derived from KdV by Miura transform, and a number of properties regarding its solutions are inherited from KdV theories \cite{KP, miura}. 

For gKdV with general polynomial nonlinearity, local well-posedness in $H^s$ is known \cite{I3} for $s\geq \f{1}{2}$.  For global well-posedness, monomial-type nonlinearities $P(u) = u^{k+1}$ are considered and the distinction between focusing and defocusing cases become more significant due to supercricality.  For defocusing cases, global well-posedness in $H^s$ is known \cite{Bao} for $s\geq \f{1}{2}$ for $k=3$ and for $s> \f{5}{9}$ for $k=4$.  Furthermore for the supercritical regime (i.e. $k>4$), defocussing gKdV is known \cite{I3} to be globally well-posed in $H^s$ for $s> \f{13}{14} - \f{2}{7k}$.  In \cite{blowup}, focusing mass-critical gKdV on the real-line is shown to contain finite-time blow-up solutions in $H^s$, while we are unaware of  analogous blow-up result   for the torus.

Another feature, which we will be a central theme of interest for us in this article,
concerns nonlinear smoothing property.  Informally, nonlinear smoothing refers to the idea that the roughest part of a solution is contained in the free solution (in some cases, another explicit expression of initial data \cite{fnls, Oh}), whereas the remainder is smoother than the initial data.  These issues have been thoroughly explored in many contexts, see \cite{titi, linear,  erdogan, japanese, Oh, TT}. 

Either directly or indirectly, nonlinear smoothing property has been one of the main ingredients in establishing polynomial bounds for $H^s$ norm of solutions.  Quest for obtaining polynomial bounds was initiated by Bourgain's work in \cite{bour3} and was followed by \cite{I2, erdogan, fnls, Soh1, Soh2, St1, St2, St3} among others.  Results concerning higher Sobolev norms demonstrate low-to-high frequency cascade in solutions and has implications to the physical phenomenon of weak-turbulence.  See \cite{Soh1} and references there in.

Results on lower bounds for Sobolev norms of solutions are more scarce.  In \cite{bour4}, a perturbation of a nonlinear wave equation is shown to contain solutions whose Sobolev norms grow polynomially in time.  Also, authors in \cite{I4} demonstrated that 2D cubic nonlinear Schr\"odinger equation on the torus contains solutions which grows over time, although the rate of growth is not necessarily a polynomial-type.    

For many completely integrable models, uniform-in-time bounds are available.  For instance, $H^s$~norm of solutions of KdV is shown \cite{KVZ, KT} to be uniformly bounded in time for $s\geq -1$ and also $s>-\f{1}{2}$ in case of mKdV.  These results take advantage of inverse scattering technique and do not extend to related non-integrable models.  For this reason, polynomial-bound results are still relevant for non-integrable perturbations of KdV and mKdV using linear potentials or variable coefficients \cite{erdogan, St2}.
\subsection{Main results}
Going back to \eqref{eq:gKdV}, let us make some reductions, which will be helpful in the sequel. First,  the solutions to \eqref{eq:gKdV} (formally) conserve their  mean value 
\begin{equation}
\label{102}
\int_\bt u(t,x)\, dx = \int_\bt f(x) \, dx.
\end{equation}

The transformation $u(t,x) \mapsto u(t,x) - \int_\bt f(x')\,dx'$ changes RHS of \eqref{eq:gKdV} to be $\p_x (P(u+ \int f\,dx'))$ which still belongs to the class of derivative polynomial nonlinearity.  Thus, we may assume that the solution $u$ of \eqref{eq:gKdV} has the mean-zero property.  Further, it is our standing assumption that the solution is real-valued. Additionally, the transformation $u(t,x) \mapsto u(t,x-ct)$ can eliminate any linear term in $P(u)$ and any constant term is removed by the derivative.  So, we may assume that the smallest possible degree monomial appearing in $P(\cdot)$ is a quadratic term.  

Below, we state local well-posedness of \eqref{eq:gKdV} for $H^{\f{1}{2}+}(\bt)$ data. 
\begin{theorem}\label{th:lwp}
	Let $s> \f{1}{2}$ and $0< \ve \leq s-\f{1}{2}$. Then there exists $T = T(\n{f}{H^{\f{1}{2}+\ve}})>0$ such that, the equation \eqref{eq:gKdV} has a unique solution $u$ given as a translate of  $\wt{u} \in Y_T^s\hookrightarrow C^0_t ([0,T]; H^s_x(\bt))$ given in \eqref{eq:u2}.  Here $Y_T^s$ is an auxiliary space to be defined in Section~\ref{sec:not}.\\
	
	Furthermore, this solution~$u$ of \eqref{eq:gKdV} satisfies
	\[
\n{u}{C^0_t([0,T]; H^s_x} \lesssim 	\n{\wt{u}}{Y^s_T} \leq C(P, s, \n{f}{H^{\f{1}{2}+\ve}}) \n{f}{H^s}.
	\]
	In particular, if we assume an a priori control of $H^1$ norm, then this solution extends globally in time for $s\geq 1$.
\end{theorem}

Local well-posedness of \eqref{eq:gKdV} in $H^s$ for $s\geq \f{1}{2}$ was proved in \cite{I3}.  Theorem~\ref{th:lwp} does not contain the endpoint $s=\f{1}{2}$ but is otherwise sharp in the sense of analytic well-posedness.  

We now redirect our discussion to the smoothing property, which is given below. 
\begin{theorem}\label{th:smoothing}
	Let $s\geq \f{1}{2}$ and $u$ be the solution given in Theorem~\ref{th:lwp} and let $\wt{u}$ be a translate of $u$ given in \eqref{eq:u1}.  Then for each $0<\ga <\min (2s-1, 1)$, there exists $T =T(\n{f}{H^{\f{1}{2}+\ve}})>0$ satisfying 
	\[
	\n{\wt{u}(t) - e^{t\p_x^3} f}{C_t^0 ([0,T]; H^{s+\ga}_x(\bt))} \leq C(\ga, P, \n{f}{H^{\min(s,1)}}) \n{f}{H^s} \qquad \textnormal{ for all } 0\leq t \leq T. 
	\]
\end{theorem}
Statement of Theorem~\ref{th:smoothing} is in the same vein as results obtained in \cite{fnls, Oh} where smoothing is obtained after imposing a resonant phase-shift on either the nonlinear solution or the free solution.  In this case, the phase-shift \eqref{eq:u1} is simply a translation which is invertible and well-behaved.  As in \cite{erdogan, Oh},  such smoothing property can be shown to persist globally in time if $H^s$~norm of the underlying solution does not blow up in  finite time.  In \cite{erdogan, Oh}, KdV is shown to have a smoothing of order $1-\ve$ for $f\in H^s$ for $s > -\f{1}{2}$.  Theorem~\ref{th:smoothing} demonstrates that such smoothing effect is shared by gKdV in higher regularity, $s\geq 1$.  Low-regularity smoothing is unavailable for gKdV due to  absence of required Strichartz estimates.

Our next topic concerns polynomial bounds. Clearly, in order to discuss polynomial bounds, one needs global solutions. 
Our work does not present new results in this direction, instead we focus on models in the form \eqref{eq:gKdV} for which $H^1$ norm of a solution is a priori controlled.  In this case, Theorem~\ref{th:lwp} guarantees global solutions in $H^s$ for $s\geq 1$.

We describe the framework in a bit more detail.  In addition to the conservation law \eqref{102}, there is conservation of mass and Hamiltonian
\begin{equation}
\label{105}
\begin{array}{l} 
I[u(t)] =\int_\bt u^2(t,x)\, dx = I[f], \\[10pt]
H[u(t)] =\int_{\bt} \f{1}{2} |u_x|^2 + G(u(x))\, dx=H[f].
\end{array}
\end{equation}
where $G$ is the polynomial with $G'(z)=P(z), G(0)=0$.  Conservation of $H$ is especially important, as it sometimes allows for control of $\|u(t, \cdot)\|_{H^1}$ along the evolution. It is also well-known  that, due to the Gagliardo-Nirenberg-Sobolev inequality,  a priori control of $H^1$ norm is automatic for all $H^1$ initial data, if $deg(P)\leq 4$, or if $deg(P)=5$ and the data is small.  Even when $P$ contains higher-power nonlinearities, conservation of $H$ may still provide the $H^1$ bounds, provided  $G$ is of even power with a positive leading  coefficient. This leads to global well-posedness for these models. On the other hand, the equations \eqref{eq:gKdV} containing  higher power nonlinearities lack such control\footnote{specifically when  $G$ is of odd power or alternatively, $G$ is of even power, but with negative   leading order coefficient} and consequently, they may exhibit finite time blow-up \cite{blowup}. 
Here is our final main result, which provides polynomial-in-time bounds for global solutions of \eqref{eq:gKdV}. 
\begin{theorem}\label{th:main}
Let $s>1$ and assume an a priori control of $H^1$ norm of solutions for \eqref{eq:gKdV}. Then the global-in-time solution $u$ given in Theorem~\ref{th:lwp} satisfies the following polynomial-in-time bound:
\begin{equation}
\label{645}
\n{u(t)}{H^{s}} \leq C(\ve, P, \n{f}{H^s}) \lan{t}^{s-1+ \ve} \qquad \textnormal{ for any } \ve >0.
\end{equation}
\end{theorem}
Per the discussion preceding the statement of Theorem \ref{th:main}, we may state the following representative   corollary. 
\begin{corollary}
	\label{Cor:1} 
	Suppose the nonlinearity is given by $P(z)=\sum_{j=2}^{2N+1} a_j z^j$, where $a_{2N+1}>0$. Then, for $f\in H^s(\bt), s>1$, the unique global solution to \eqref{eq:gKdV} obeys the polynomial bound \eqref{645}. 
\end{corollary}

{\bf Remarks:} 
\begin{itemize}
\item  For KdV, mKdV and Gardner models, uniform-in-time bounds are available \cite{KT}.  

\item In \cite{St2}, a polynomial bound with the same exponent as Theorem~\ref{th:main} is shown for non-integrable perturbation of KdV and mKdV on the torus.  For gKdV with monomial nonlinearity $P(u) = u^{k+1}$, a polynomial bound is derived \cite{St3} at rate $\lan{t}^{2s}$ for $s\geq 1$, assuming a priori control of $H^1$ norm.  Theorem~\ref{th:main} improves the exponent for gKdV bounds to the level equivalent to perturbed KdV and mKdV models given in \cite{St2}.

\item Our scheme suggests that statement of Theorem~\ref{th:main} may extend to the class of nonlinearities given by $C^\infty$ functions $P$, which are analytic at zero.  We do not pursue this herein.
\end{itemize}

We now outline the plan for the paper. For simplicity of our subsequent discussion, we work with  polynomials containing only  two terms: namely $P(u) = a u^n + b u ^m$.  It will be apparent that our scheme  easily extends to  general polynomial models. In Section~\ref{sec:not}, we introduce basic notations and functions spaces, as well as linear estimates from literature. In Section~\ref{sec:mult}, we present one of the main technical tools, namely the multi-linear  estimates, which allow for the smoothing estimates later on. In Section \ref{Sec:3}, we provide the proof of Theorem \ref{th:lwp}.  In Section \ref{sec:4}, we perform  normal form transformation of the equation and  prove consequent smoothing estimates. In Section~\ref{sec:5}, we prove the nonlinear smoothing property given in Theorem~\ref{th:smoothing}. Finally, in Section~\ref{sec:6}, we introduce a new technique for deriving polynomial bounds from nonlinear smoothing estimates. We believe that this simple and efficient method is new and it is likely to be useful in other situations as well.

\section{Functional spaces and Equations set-up}
\label{sec:not}

For any function $f(x)$, denote the $k$~th Fourier coefficient by $f_k$ for $k\in \bz$.  If the function has the mean-zero property, then we assume that $k\neq 0$.  When as sum over such index is written, it is assumed that the summation takes place over $k\in \bz^* :=\bz\setminus\{0\}$ rather than over $\bz$.  Also, for a function of $t$, denote the Fourier transform of $g(t)$ by $\wh{g}(\tau)$.

We denote $\lan{\cdot} = (1+ \cdot)^{\f{1}{2}}$.

For any two quantities $A$ and $B$, we write $A\lesssim B$ (similarly $A\gtrsim B$) if there is an absolute positive constant $c$ satisfying $|A| \leq c |B|$ (similarly $|A| \geq c |B|$).  Negation of $A\lesssim B$ is denoted $A \gg B$ (similarly, negation of $A\gtrsim B$ is $A\ll B$).  Concurrence of $A\lesssim B$ and $A \gtrsim B$ is denoted $A\sim B$.
\subsection{Functions spaces} Let $T>0$.  For any functional space $Y\hookrightarrow C_t^0$ involving time variable $t$, define the norm $Y_T$ by  
\[
\n{u}{Y_T} = \inf\{\n{v}{Y}: v\in Y \tn{ and } v(t) = u(t) \tn{ for all } 0\leq t \leq T\}.
\]
Next, we introduce Bourgain space $Y^s := X^{s,\f{1}{2}} \cap H^s_x L^1_{\tau}$ defined in \cite{I1} which embeds in $C^0_t H^s_x$.  This norm is defined by
\[
\n{u}{Y^s} = \n{\lan{\tau - k^3}^{\f{1}{2}} \lan{k}^s \wh{u_k}(\tau)}{L^2_\tau l^2_k} +\n{\lan{k}^s \wh{u_k}(\tau)}{ l^2_k L^1_{\tau}}.
\]
Because $X^{s,\f{1}{2}}$ fails to embed in $C^0_t H^s_x$ and also $X^{s,b}$ bilinear estimate fails for any $b \neq \f{1}{2}$ \cite{KPV}, it is necessary to intersect it with $H^s_x L^1_\tau$.  As give in \cite{I1}, the space $Y^s$ is accompanied by the space $Z^s$ whose norm is defined by 
\[
\n{u}{Z^s} =  \n{\lan{\tau - k^3}^{-\f{1}{2}} \lan{k}^s \wh{u_k}(\tau)}{L^2_\tau l^2_k} +\n{\f{\lan{k}^s \wh{u_k}(\tau)}{ \lan{\tau - k^3}  }}{ l^2_k L^1_{\tau}}.
\]
Then, we state the following results:

\begin{proposition}\cite{I1} \label{pro:iteam}
For any $\eta \in \mathcal{S}_t$ and $f \in C^{\infty}_x(\bt)$, $F \in Z^s$, 
\begin{align*}
\n{ \eta(t) e^{t\p_x^3} f}{Y^s} &\lesssim \n{f}{H^s}\\
\n{\eta(t) \int_0^t e^{(t-s) \p_x^3} F(s)\, ds}{Y^s}  &\lesssim \n{F}{Z^s}.
\end{align*}

\end{proposition}

Following linear estimates are known from \cite{Bour}:
\begin{align*}
\n{\eta(t) v}{L^2_t L^2_{x}} &= \n{v}{X^{0,0}},
\n{\eta(t) v}{L^\infty_t L^2_{x}} \lesssim_\de \n{v}{X^{0,1/2 +\de}},\\
\n{\eta(t) v}{L^4_{x,t}} &\lesssim \n{v}{X^{0,1/3}},
\n{\eta(t) v}{L^6_{x,t}} \lesssim_{\ve, \de} \n{v}{X^{\ve, \f{1}{2}+\de}},\\
\n{\eta(t/T) v}{X^{s,b}} &\lesssim T^{b'-b} \n{v}{X^{s,b'}} \qquad \textnormal{ for } -\f{1}{2}<b<b'<\f{1}{2}.
\end{align*}
Interpolation yields  that, given any $2<q<6$, there exist some  $\ve>0$ satisfying
 \[
\n{\eta(t) v}{L^{q}_{x,t}} \lesssim_{\ve} \n{v}{X^{\ve, \f{1}{2}-\ve}}.
\]

Following linear estimates are from \cite[Lemma 2.2]{Bao}
\begin{align*}
\n{\eta(t) v}{L^\infty_{t,x}} &\lesssim_\ve \n{v}{Y^{\f{1}{2}+\ve}},\\
\n{\eta(t) v}{L^6_{t,x}} &\lesssim_\ve \n{v}{Y^{\ve}}.
\end{align*}

\subsection{Equation set-up}
Consider the nonlinearity in \eqref{eq:gKdV}, which we assumed that it consists of two terms only.   We can decompose the $k$~th Fourier coefficient of $\p_x (u^n + u^m)$ as
\[
a \sum_{\tiny k_1 + \cdots + k_n = k} i k \prod_{j=1}^n u_{k_j} +b \sum_{\tiny k'_1 + \cdots + k'_m = k} i k \prod_{j'=1}^m u_{k_{j'}}  
\]
Resonance in gKdV occurs when one of the interior frequency~$k_j$ equals the exterior frequency~$k$.  Note that a quadratic nonlinearity does not contain any resonance due to the mean-zero restriction.  The RHS can be written as $\rr[u]+\nr[u]$ where
\[
\rr[u] := a \sum_{\tiny\begin{array}{c} k_1 + \cdots + k_n = k\\ k_{j_0} = k \textnormal{ for some } j_0 \end{array}} i k \prod_{j=1}^n u_{k_j}+ b  \sum_{\tiny\begin{array}{c} k'_1 + \cdots + k'_m = k\\ k_{j'_0} = k \textnormal{ for some } j'_0 \end{array} }i k \prod_{j'=1}^m u_{k'_j}.
\] 
Here $\rr$ stands for Resonant term and $\nr$ stands for Non-Resonant term.
Let us carefully examine the structure of $\rr[u]$.  We will focus on the first term corresponding to $a \p_x(u^n)$ since the second term can be managed analogously.  The set over which this summation occurs is 
\[
R_k = \{(k_1, \cdots, k_n)\in (\bz^*)^n: k_1 +\cdots+ k_n = k, \, k_{j_0} = k \tn{ for some } 1\leq j_0\leq n\}.
\]
Define for $1 \leq l \leq n$,
\[
R_k^l = \{(k_1, \cdots, k_n)\in (\bz^*)^n: k_1 +\cdots+ k_{l-1} + k_{l+1}+ \cdots+ k_n = 0, \, k_{l} = k\}.
\]
Clearly, $R_k = \ds \cup_{l=1}^n R_k^l$ but there are a lot of repeated elements in this union.  We will need to keep track of these repeated elements later.  First, we note that 
\[
\sum_{R_k^l} i k \prod_{j=1}^n u_{k_j}  =ik u_k \sum_{\tiny k_1 +\cdots+ k_{l-1} + k_{l+1}+ \cdots+ k_n = 0 } \prod_{j\neq l} u_{k_j}  =   ik u_k \int_\bt u^{n-1}(t,x)\,dx 
\]
which is independent of $l$.  Finally, observe that
\[
\sum_{R_k} = \sum_{l=1}^n \sum_{R_k^l} - \sum_{l_1, l_2} \sum_{R_k^{l_1}\cap R_{k}^{l_2}} + \sum_{l_1, l_2, l_3} \sum_{R_k^{l_1} \cap R_k^{l_2} \cap R_k^{l_3}} - \cdots.
\]
We will show that (1) the first summation can be eliminated by a well-behaved transformation and (2) the remaining summation already has enough smoothness already built in.  

We write $\rr[u] = \rr^1[u] + \rr^2[u]$ where $\rr^1[u]$ contains only the first summation.  That is,
\[
\rr^1[u]_k =  ik u_k \left(a n \int_\bt u^{n-1}(t,x)\,dx+ bm  \int_\bt u^{m-1}(t,x)\,dx\right).
\]
Consider a transformation defined by 
\begin{equation}\label{eq:u1}
\wt{u}_k(t) := u_k (t) \exp\pr{ik \int_0^t \int_\bt a n u^{n-1}  (t',x) + b m u^{m-1}  (t',x)\, dx \, dt'}.
\end{equation}
Note that $u$ and $\wt{u}$ has the same initial value, $\n{u}{H^s_x} = \n{\wt{u}}{H^s_x}$ and also \\ $\int_\bt u^{p} (t,x)\, dx = \int_{\bt} \wt{u}^{p} (t,x)\, dx$ for any $p\in \bn$.  If $u$ solves \eqref{eq:gKdV}, then $\wt{u}$ solves 
\begin{equation}\label{eq:util}
\left|\begin{array}{l}(\p_t+\p_{xxx}) \wt{u} = \rr^2[\wt{u}] + \nr[\wt{u}], \\ \wt{u}\vert_{t=0} = f \in H^s(\bt).\end{array}\right.
\end{equation}
Conversely, if $\wt{u}$ solves above, then
\begin{equation}\label{eq:u2}
u(t,x) = \wt{u} \pr{t, x+ \int_0^t \int_\bt an \wt{u}^{n-1} (t',x) + bm  \wt{u}^{m-1} (t',x)\, dx\, dt'}
\end{equation}
satisfies \eqref{eq:gKdV}. Thus, we will focus on solving \eqref{eq:util} from here on.  For simplicity of notation, we will still use $u$ instead of $\wt{u}$ in the sequel.
\section{Symmetric multi-linear estimates}\label{sec:mult}

In order to establish multi-linear dipersive estimates, we first consider nonlinear dispersive interactions:  For any $n\in \bn$, denote
\[
H_n := H_n (k_1, \cdots, k_n) =  \left( \sum_{j=1}^n k_j\right)^3 - \sum_{j=1}^n k_j^3.
\]
We need a result which exploits the dispersion relation to its fullest.  We will however convince ourselves  that this  alone will not give the desired smoothing estimates, without the help of a normal form transformation. Nevertheless, this will reduce matters to some very specific terms, which will be eliminated via the said normal forms. 
\subsection{Analysis of the dispersion relation}
We have the following Proposition. 
\begin{proposition}
	\label{pro:disp}  
Consider $k= \sum_{j=1}^n k_j$ for $k_j \in \bz^*$ and $H_n$ as defined above.  Denote $k_{\max} := \max\{|k_1|,\cdots, |k_n|\}$ and $k_{\max_{j}}$ be the $j^{\textnormal{th}}$ largest term in $\{|k_1|, \cdots, |k_n|\}$.
\begin{enumerate}
\item Let $n = 2$.  Then A: $H_2 \gtrsim k_{\max}^2$.
\item Let $n = 3$.   At least one of the following is true:
\begin{enumerate}[A.]
\item $H_3 \gtrsim k_{\max}^2$.
\item $k_{j_0} = k$ for some $j_0 \in \{1,2,3\}$.
\item $k_j \gtrsim k$ for all $j\in \{1,2,3\}$.
\end{enumerate}
\item Let $n \geq 4$.  At least one of the following is true:
\begin{enumerate}[A.]
\item $H_n \gtrsim k_{\max}^2$.
\item $k = k_{j_0}$ for some $j_0\in \{1,\cdots, n\}$. (resonance)
\item $k_{\max_3} \gtrsim k$.
\item $k_{\max_3}^2 k_{\max_4} \gtrsim k_{\max}^2$.
\end{enumerate}
\end{enumerate}
\end{proposition}
\begin{proof}
For $H_2$, the proof is immediate when since $H_n = 3 k k_1 k_2$ where none of the factors can vanish.  At least two of the frequencies should be comparable to $k_{\max}$, giving the desired result.

For $n\geq 3$, we will derive a contradiction after negating all conditions listed above.  For $n=3$, we assume by contrary that
\begin{enumerate}
\item $H_3 \ll k_{1}^2$.
\item $k \neq k_{j}$ for any $j\in \{1,2,3\}$.
\item $k_{3} \ll k$.
\end{enumerate}
Note that $H_3 = 3 (k_1 + k_2)(k_2+ k_3)(k_3+ k_1)$.  None of the factors can vanish due to the assumption (2).  Our assumption $k_3 \ll k \lesssim k_1$ implies $k_3 + k_1 \sim k_1$.  Also, if $k_1 \sim k_2$, then the second factor $k_2 + k_3 \sim k_1$; and if $k_1 \gg k_2$, then the first term $k_1 + k_2 \sim k_1$.  In either cases, we get $H_3 \gtrsim k_1^2$ which contradicts with our assumption.  This proves our estimate for $H_3$.

For $n\geq 4$, we define the following notation: 
\[
\wt{k_j} :=  k_j + \cdots + k_n \qquad \tn{where }|k_1| \geq |k_2| \geq \cdots \geq |k_n|
\]
where we have assumed a descending order of frequency indices without loss of generality. we assume that
\begin{enumerate}
\item $H_n \ll k_{1}^2$.
\item $k \neq k_{j}$ for any $j\in \{1,\cdots, n\}$.
\item $k_{3} \ll k$.
\item $k_{3}^2 k_{4} \ll k_{1}^2$.
\end{enumerate}
We begin by observing the following identity:
\begin{align}
H_n &= \wt{k_1}^3 - k_1^3 - \cdots - k_n^3 \notag\\
	&= (\wt{k_1}^3 - k_1^3 - \wt{k_2}^3) + (\wt{k_2}^3 - k_2^3 - \wt{k_3}^3)+ \cdots + (\wt{k_{n-1}}^3 - k_{n-1}^3 - \wt{k_n}^3)\notag\\
	&= 3 k_1 \wt{k_1} \wt{k_2}+ 3 k_2 \wt{k_2} \wt{k_3} + \cdots +  3 k_{n-1} \wt{k_{n-1}} \wt{k_n}\label{eq:hn}\\
	&= 3 \sum_{j=1}^{n-1}  k _j\wt{k_j} \wt{k_{j+1}}. \notag
\end{align}
Consider the first two terms of $H_n$ from above:
\[
k_1\wt{k_2} \wt{k_1} +  k_2 \wt{k_2}  \wt{k_{3}} = \wt{k_2}(\wt{k_3} + k_1)(k_1 + k_2).
\]  
Assumption $k_3 \ll k$ implies that $k_1 + k_2 \neq 0$.   Also, since $\wt{k_3} \lesssim k_3 \ll k \lesssim k_1$, we have that $k_1 + \wt{k_3} \neq 0$.   Finally, we are given by assumption that $\wt{k_2}\neq 0$.  Thus, the expression above does not vanish.

Since $\wt{k_3} \ll k_1$, the middle factor $\wt{k_3} + k_1 \sim k_1$.  We claim that the remaining term $\wt{k_2} (k_1 + k_2) = (k_2 + \wt{k_3})(k_1+ k_2) \gtrsim k_1$.  To see this, we need to split into two cases: $k_1 \sim k_2$ or $k_1 \gg k_2$.

In the first case $k_1 \sim k_2$, the first factor $k_2 + \wt{k_3}\sim k_1$ since by assumption $\wt{k_3} \ll k_1 \sim k_2$.  In the second case $k_1 \gg k_2$, the last factor $k_1 + k_2 \sim k_1$.  This shows that, in either cases, $ (k_2 + \wt{k_3})(k_1+ k_2) \gtrsim k_1$.

So far, we have shown that the first two terms of $H_n$ is at least the size of $k_1^2$.  Since we need to have $H_n \ll k_1^2$, we need the remaining terms of $H_n$ to be comparable to $k_1^2$ in order to cancel out the first two terms.  Namely we need
\[
k_3 \wt{k_3} \wt{k_4} + \cdots + k_{n-1} \wt{k_{n-1}} \wt{k_n} \gtrsim k_1^2.
\]
But note $\wt{k_j} \lesssim k_j$ for any $j$, so we have 
\[
k_1^2 \lesssim k_3 \wt{k_3} \wt{k_4} + \cdots + k_{n-1} \wt{k_{n-1}} \wt{k_n} \lesssim k_3^2 |k_4| + \cdots + k_{n-1}^2 |k_n| \lesssim k_3^2 k_4.
\]
But our assumption states $k_3^2 k_4 \ll k_1^2$, which contradicts with above.  This proves the desired result.
\end{proof}
{\bf Remark:} 
Proposition~\ref{pro:disp} above is used to establish our heuristics.  It would be helpful to establish a strategy at this point using a rough derivative count.  Our goal is to prove an estimate of the form:
\[
\n{\p_x (u^n)}{Z^{s+\ga}} \lesssim \n{u}{Y^s}^n.
\]
This means that we need to fight $1+s+ \ga$~derivatives with $|H_n|^{1/2}$ as well as $\prod_{j=1}^n |k_j|^s$.  Case B is resonance, which is mostly taken care of via direct transformation.  We still need to deal with $\rr^2$, but at least three internal frequencies are comparable to $k$ in $\rr^2$.  This places $\rr^2$ in Case C.

In Case C, we have that $(k_{\max_1} k_{\max_2} k_{\max_3})^s \gtrsim |k|^{3s}$, which gives us $3s$ derivatives without using any of the dispersive gain~$H_n$.  So we need $1+s+\ga < 3s \iff \ga < 2s -1$.

In Case D, $(k_{\max_1} k_{\max_2} k_{\max_3} k_{\max_4})^s \gtrsim k_{\max_1}^s  (k_{\max_3}^2 k_{\max_4})^s  \gtrsim k_{\max}^{3s}$.   This leads to the same restriction as above.

In Case A, a normal form must be used because the gain of $H_n^{1/2} \gtrsim k_{\max}^1$ is insufficient for any $\ga >0$.  But taking a normal form for this term means that we gain the full $H_n$ derivative.  The remainder terms will now contain an extra derivative, which means that we now need to fight $2+s +\gamma$ derivatives.  Establishing remainder estimates for normal form will be very delicate and this is as far as heuristics can take us.

We now need a technical tool to estimate various operator norms of multi-linear operators, based on size estimates of the multpliers. 

\subsection{Size estimates for multi-linear operators}
\begin{definition}
Given a symbol $\si  = \si(k_1, \cdots, k_n)$ with $n\geq 2$, a multi-linear Fourier multiplier $T^n_\si$ is defined via 
\begin{equation}\label{eq:tsi}
T^{n}_\si (u_1, \cdots, u_n) :=\sum_{k\in \bz^*} \sum_{\ (k_1, \cdots, k_n) \in \Om_k } \si(k_1, \cdots, k_n) \prod_{j=1}^n (u_j)_{k_j} e^{i kx}  
\end{equation}
where $\Om_k$ places a restriction in frequency interactions.  This domain is associated with the operator $T^n_\si$.
\end{definition}

Following is one of the main tools in achieving our estimates.  It deals with symmetric estimates where all inputs of $T^n_\si$ are identical.

\begin{proposition}\label{pro:main}
Let $n\geq 2$ and $0<T \ll 1$.  Consider estimates of the following type:  Given a multi-linear operator $T^n_\si$ as defined in \eqref{eq:tsi}
\[
\n{T^n_\si (u, \cdots, u)}{Z^{s_0}} \lesssim_{\ve, n} T^{\ve} \n{u}{Y^{s_1}}\n{u}{Y^{s_2}}^{n-1}
\]
for some $s_0,s_1, s_2 > \f{1}{2}$ where $u := u(t) = \eta (t/T) u(t)$ for a smooth cut-off function $\eta$.  The inequality above is satisfied if either of the following conditions is met.
\begin{itemize}
\item If dispersion weight $H_n$ is used, then we need
\begin{equation}\label{eq:prop1}
\sup_{(k_1, \cdots, k_n) \in \Om_k} \f{ |k|^{s_0} |\si|(k_1,\cdots, k_n)}{ \lan{H_n}^{\f{1}{2}} k_{\max_1}^{s_1} k_{\max_2}^{s_2}}  = O(1)
\end{equation}
where $k_{\max_j}$ is as defined in Proposition~\ref{pro:disp}. \\

\item If dispersion weight $H_n$ is not used, then we need
\begin{equation}\label{eq:prop2}
\sup_{(k_1, \cdots, k_n) \in \Om_k} \f{ |k|^{s_0} |\si|(k_1,\cdots, k_n)}{ k_{\max_1}^{s_1-\ve} \left(k_{\max_2} k_{\max_3} k_{\max_4}\right)^{s_{2}}}  = O(1).
\end{equation}
We use the convention that $k_{\max_j} = 1$ for $j>n$.
\end{itemize}

\end{proposition}

\begin{proof}
\textbf{Assuming condition \eqref{eq:prop1}}:  We need to show two estimates: one for $X^{s_0,-\f{1}{2}}$, and the other for $l^2_k L^1_\tau$ with a given weight. By splitting the frequency set $\Om_k$ into $n!$~rearranged partitions, we can assume $|k_1| \geq |k_2| \geq \cdots.$

Let the symbol above be bounded by $M>0$.  First we show the estimate for $X^{s,-\f{1}{2}}$. 
\[
\n{ T^n_\si(u, \cdots, u)}{X^{s_0, - \f{1}{2}}} = \sup_{\n{z}{X^{-s_0,\f{1}{2}}} = 1} \int_{\bt \times \br}  T^n_\si(u, \cdots, u) \, z   \, dx\, dt
\]
By Plancherel, the integral on the RHS is bounded by 
\[
\int_{\Gamma} \sum_{(k_1, \cdots, k_n) \in \Om_k}  \si(k_1, \cdots, k_n) \wh{z_k}(\tau)  \prod_{j=1}^n \wh{u_{k_j}}(\tau_j) \, d\Gamma
\]
where $d\Gamma$ is the inherited measure on the hyperplane $\Gamma$ given by
\[
\Gamma = \{ (\tau, \tau_1, \cdots, \tau_n, k, k_1,\cdots, k_n): \tau_1 + \cdots + \tau_n = \tau, k_1 + \cdots + k_n = k\}.
\]
Note that the summand above is controlled by
\begin{equation}\label{eq:pf1}
\lan{H_n}^{\f{1}{2}} |k_1|^{s_1} |k_2|^{s_2} |k|^{-s_0}  \wh{z_{k}}(\tau) \prod_{j=1}^n \wh{u_{k_j}}(\tau_j).
\end{equation}
By algebraic association, we note that
\[
H_n^{1/2} \lesssim  \sum_{j=0}^{n} \lan{\tau_j - k_j^3}^{1/2}
\]
where we denoted $\tau_0 := \tau$ and $k_0 := k$.  Then we can replace $H_n$ by the sum above.  The first term from \eqref{eq:pf1} is
\[
\lan{\tau_0 - k_0^3}^{1/2}|k_0|^{-s_0}  \wh{z_{k_0}}(\tau_0)  |k_1|^{s_1}\wh{u_{k_1}}(\tau_1) |k_2|^{s_2}\wh{u_{k_2}}(\tau_2)\prod_{j=3}^n \wh{u_{j,k_j}}(\tau_j).
\]
Although the argument is not totally symmetric, the other terms are not so different.  We thus omit the other terms (i.e. ones containing $\lan{\tau_j - k_j^3}$ for $j>0$).   Applying Plancherel, we apply H\"older's inequality, we place $L^2_{t,x}$~norm on $z$, $L^4_{t,x}$ on the next two terms and $L^{\infty}_{t,x}$ on the rest.  We obtain the bound which is
\[
\n{z}{X^{-s_0,\f{1}{2}}} \n{D_x^{s_1} u}{L^4_{t,x}} \n{D_x^{s_2} u}{L^4_{t,x}} \n{u}{L^\infty_{t,x}}^{n-2}.
\]
 By linear estimates, this is bounded by
 \[
\n{z}{X^{-s_0, \f{1}{2} }} \n{u}{X^{s_1,\f{1}{3}}} \n{u}{X^{s_2,\f{1}{3}}} \n{u}{Y^{s_2}}^{n-2}
\]
for any $s_2>\f{1}{2}$.  Note that we have a room spare in the $X^{s,b}$~weight.  Using time-localization, we can obtain a positive power in $T$ for the $X^{s,-\f{1}{2}}$ bound.
 
Next we need to estimate
\begin{equation}\label{eq:pf2}
\n{\f{\lan{k}^{s_0} \cf_{t,x}[T^n_\si]}{\lan{\tau-k^3}}}{l^2_k L^1_\tau}
\end{equation}
We will split into two cases: first when $\lan{\tau-k^3} \sim H_n$ and second when $\lan{\tau-k^3} \not\sim H_n$.

In the first case, Note that $\lan{\tau -k^3}^{-\f{1}{2}} \chi_{H_n}(\lan{\tau-k^3}) \in L^2_\tau$ uniformly in $k$ and $H_n$. Then \eqref{eq:pf2} under this restriction is bounded by
\[
\n{\f{\lan{k}^{s_0} \cf_{t,x}[T^n_\si]}{\lan{\tau-k^3}^{1/2}}}{l^2_k L^2_\tau} = \n{T^n_\si}{X^{s_0,-\f{1}{2}}}
\]
which is bounded by the RHS as before.

Next, if $\lan{\tau-k^3} \not\sim H_n$, then there is a $j_0 \in \{1,\cdots, n\}$ satisfying either $\lan{\tau_{j_0} - k_{j_0}^3} \gtrsim H_n$.   Using Cauchy-Swartz, \eqref{eq:pf2} is bounded by
\[
\n{\lan{\tau-k^3}^{-\f{1}{3}} \lan{k}^{s_0} \cf_{t,x}[T^n_\si]}{l^2_k L^2_\tau} = \n{T^n_\si}{X^{s_0,-\f{1}{3}}} = \sup_{\n{z}{X^{-s_0,\f{1}{3}}}} \int_{\Ga}\sum_{\Om_k} \si \, \prod_{j=1}^n \wh{u_{k_j}}\, \wh{z_k}\, d\Ga.
\]

Note that, in this case, $\lan{H_n}^{\f{1}{2}}$ can be replaced by $\sum_{j=1}^n \lan{\tau_j - k_j^3}^{\f{1}{2}}$.  Once again, we just take the first term of this sum:  In this case, we need to estimate
\[
 |k_0|^{-s}  \wh{z_{k_0}}(\tau_0) \lan{\tau_1 - k_1^3}^{\f{1}{2}} |k_1|^{s_1} \wh{u_{k_1}}(\tau_1) |k_2|^{s_2}   \wh{u_{k_2}}(\tau_2) \prod_{j=3}^n \wh{u_{k_j}}(\tau_j).
\]
Applying Plancherel and H\"older, we place $L^2_{t,x}$ on the middle term, $L^4_{t,x}$ on the first and the third term and $L^{\infty}_{t,x}$ on the rest. Then linear estimates give the bound
\[
\n{z}{X^{-s_0, \f{1}{3}}} \n{u}{X^{s_1, \f{1}{2}}} \n{u}{X^{s_2, \f{1}{3}}}  \n{u}{Y^{s_2}}^{n-2}
\]
where $X^{s_2,\f{1}{3}}$ can yield a positive power in $T$ by time localization.  This shows the desired result with assumption~\eqref{eq:prop1}.\\

\textbf{Assuming condition~\eqref{eq:prop2}:}  Here we do not need to use $H_n$, which makes the arguments simpler.  We will prove the statement for $n\geq 4$.  For $n=2,3$, the numerology resulting from H\"older and linear estimates are strictly better and we omit these computations below.   For the $X^{s_0,-\f{1}{2}}$ estimate, we need to bound
\[
|k_0|^{-s_0}  \wh{z_{k_0}}(\tau_0)  |k_1|^{s_1-\ve} |k_2|^{s_2} |k_3|^{s_2} |k_4|^{s_2} \prod_{j=1}^n \wh{u_{k_j}}(\tau_j)
\]
Using Plancherel and H\"older, we can estimate 
 \[
\n{D_{x}^{-s_0-\ve/5} z}{L^5_{t,x}}  \n{D_x^{s_1-\ve/5} u}{L^5_{t,x}} \n{D_x^{s_2-\ve/5} u}{L^5_{t,x}}^3\n{u}{L^{\infty}_{t,x}}^{n-4}. 
 \]
Using  $L^5_{t,x}$ and $L^{\infty}_{t,x}$ embeddings, we get
\[
\n{z}{X^{-s_0, \f{1}{2}-\ve}}  \n{u}{X^{s_1, \f{1}{2}-\ve}} \n{u}{X^{s_2,\f{1}{2}-\ve}}^3\n{u}{Y^{s_2}}^{n-4}. 
 \]
We can use $\n{u}{X^{s_1, \f{1}{2}-\ve}}$ to generate a small power in $T$ by time-localization.

Now for $l^2_k L^1_{\tau}$ estimate, we noted already that the case $\lan{\tau - k^3}\sim H_n$ reduces to the estimate for $X^{s_0,-\f{1}{2}}$ which is already established.  If $\lan{\tau-k^3}\sim H_n$, then we need to estimate
\[
 |k_0|^{-s_0}  \wh{z_{k_0}}(\tau_0)   |k_1|^{s_1-\ve} |k_2|^{s_2} |k_3|^{s_2} |k_4|^{s_2} \prod_{j=1}^n\wh{u_{k_j}}(\tau_j)
\]
 where $z \in X^{-s,\f{1}{3}}$.  Applying Plancherel and H\"older, we place $L^4_{t,x}$ on $z$, $L^{16/3}_{t,x}$ on four terms containing $u$ and $L^{\infty}_{t,x}$ on the rest. 
\[
\n{D_{x}^{-s_0} z}{L^4_{t,x}}  \n{D_x^{s_1-\ve/4} u}{L^{16/3}_{t,x}} \n{D_x^{s_2-\ve/4} u}{L^{16/3}_{t,x}}^3 \n{u}{L^{\infty}_{t,x}}^{n-4}. 
 \]
  By linear estimates and time-localization as before, we obtain the desired bound.
\end{proof}
The following lemma allows us to take symmetric estimates resulting from Proposition~\ref{pro:main2} and apply it to asymmetric variables as long as $s_1 = s_2$.
\begin{lemma}\label{le:asym}
Let $T_{\si}^n$ be a symmetric $n$-multi-linear operator mapping from $(Y)^n \to Z$ for some normed spaces $Y$ and $Z$.  Also, suppose that we are given
\[
\n{T_\si^n (u,\cdots, u)}{Z} \lesssim \n{u}{Y}^n \quad \tn{ for all } u \in Y.
\]
Then this implies 
\[
\n{T_{\si}^n (v_1,\cdots, v_n)}{Z} \lesssim_n \prod_{j=1}^n \n{v_j}{Y} \quad \tn{ for all } v_1, \cdots, v_n \in Y.
\]
\end{lemma}

\begin{proof}
For simplicity of notation, we will denote $T_\si^n (u) := T_\si^n (u,\cdots, u)$ when all input is identical.  Also, for any $k=1,2,\cdots, n$, let $\mathcal{P}_k$ be a set of subsets of $\{1,\cdots, n\}$ with size $k$.    We first observe the following identity:
\[
T_\si^n \pr{\sum_{j=1}^n v_j} - \sum_{\mathcal{A}\in \mathcal{P}_1} T_{\si}^n \pr{\sum_{j\not\in \mathcal{{A}}} v_j} + \sum_{\mathcal{A}\in \mathcal{P}_2} T_{\si}^n \pr{\sum_{j \not\in \mathcal{A} } v_j}- \cdots = C_n T_\si^n (v_1, v_2, \cdots, v_n).
\]
To see this identity, consider the full expansion of the first term $T_\si^n \pr{\sum_{j=1}^n v_j}$.  We would like to remove all terms from this expansion that do not include all of $v_1, \cdots, v_n$.  

Say a term from the expansion contains $v_j$ repeated $k_j$ times for $j=1,\cdots, n$ with $\sum_{j=1}^n k_j = n$.  Without loss of generality, say that $k_1 =0$: i.e. the term does not include $v_1$.  The number of occurrences of this term from expansion is the same as the number of occurrences of the identical term from expansion of $T_{\si}^n \pr{\sum_{j\neq 1} v_j}$.  Now, if $k_1$ is the only zero index, then this term will be immediately eliminated.  But if $k_1=k_2=0$, then this term is counted twice: once for $T_{\si}^n \pr{\sum_{j\neq 1} v_j}$ and another time for $T_{\si}^n \pr{\sum_{j\neq 2} v_j}$.  In this case, this repetition is canceled out by the third term $T_{\si}^n \pr{\sum_{j\neq 1,2} v_j}$.  Iterating in this manner, the identity above can be established. 

Now, add coefficients $a_j>0$ to $v_j$ which we will determine later.  Then we have
\[
T_\si^n \pr{\sum_{j=1}^n a_j v_j} + \sum_{k=1}^{n-1} (-1)^k \sum_{\mathcal{A}\in \mathcal{P}_k} T_{\si}^n \pr{\sum_{j\not\in \mathcal{{A}}} a_j v_j}  = C_n T_\si^n (v_1, v_2, \cdots, v_n).
\]
Using the given symmetric estimate, we can write
\begin{align*}
\pr{\prod_{j=1}^n a_j} \n{T_\si^n (v_1, v_2, \cdots, v_n) }{Z} &\lesssim_n \n{T_\si^n \pr{\sum_{j=1}^n a_j v_j}}{Z} + \sum_{k=1}^{n-1}\sum_{\mathcal{A}\in \mathcal{P}_k} \n{T_{\si}^n \pr{\sum_{j\not\in \mathcal{{A}}} a_j v_j}}{Z}\\
	&\lesssim_n \n{\sum_{j=1}^n a_j v_j}{Y}^n + \sum_{k=1}^{n-1}\sum_{\mathcal{A}\in \mathcal{P}_k} \n{\sum_{j\not\in \mathcal{{A}}} a_j v_j}{Y}^n\\
	&\lesssim_n \pr{\sum_{j=1}^n a_j \n{v_j}{Y}}^n + \sum_{k=1}^{n-1}\sum_{\mathcal{A}\in \mathcal{P}_k} \pr{\sum_{j\not\in \mathcal{{A}}} a_j\n{v_j}{Y}}^n.
\end{align*}
Now select $a_j = 1/\n{v_j}{Y}$.  Then RHS of above is $O_n(1)$.  Thus, dividing by $\prod_{j=1}^n a_j$, we obtain
\[
 \n{T_\si^n (v_1, v_2, \cdots, v_n) }{Z} \lesssim_n \pr{\prod_{j=1}^n a_j}^{-1} = \prod_{j=1}^n \n{v_j}{Y}.
 \]
\end{proof}

\subsection{Non-smoothing estimates}
Following is a non-smoothing estimate to obtain a priori estimate of the solution~$u$ of \eqref{eq:util} in $Y^s$.  Using the proposition established above, we only need to walk through cases A, B, C, D of Proposition~\ref{pro:disp}.  
\begin{lemma}\label{le:lwp}
For $s>\f{1}{2}$, there exists an $\ve>0$ such that for any $0< T\ll 1$,
\[
\n{\rr^2[u]\|_{Z^s} + \|\nr[u] }{Z^s} \lesssim_\ve T^{\ve} \n{u}{Y^s} \left(\n{u}{Y^{\f{1}{2}+\ve}}^{n-1} + \n{u}{Y^{\f{1}{2}+\ve}}^{m-1}\right).
\]
\end{lemma}

\begin{proof}
Here, the symbol $\si = ik$ for both $\rr^2$ and $\nr$.  In context of Proposition~\ref{pro:main}, $s_0 = s_1 = s$ and $s_2 = \min (s,1)$. \\

[Case A] Here, $H_n\gtrsim k_{\max}^2$.  Then condition \eqref{eq:prop1} is written as
\[
\f{|k|^{s} |ik|}{\lan{H_n}^{\f{1}{2}} k_{\max_1}^s} \lesssim \f{|k|^{s+1}}{k_{\max}^{1+s}}  = O(1)
\]  
which is satisfied for any $s>-1$.  Note that, for $n=2$, this is all that is required to show the desired statement.\\

[Case B]  Only $\rr^2[u]$ contains components belonging to Case B.  But in this case, we must have two internal frequencies equal to $k$.  Without loss of generality, say that $k_1 = k_2 = k$.  Then we must have $k + k_3+ \cdots + k_n = 0$, which forces $\max\{|k_3|,\cdots, |k_n|\} \sim k$.  This implies $k_{\max_3} \gtrsim k$ which makes $\rr^2[u]$ belong to Case C.  So we defer estimates for $\rr^2[u]$ to Case C.\\

[Case C] In this case, $k_{\max_3} \gtrsim k$.  Using condition \eqref{eq:prop2},
\[
\f{|k|^{s} |ik|}{k_{\max_1}^{s-\ve} (k_{\max_2} k_{\max_3})^{\f{1}{2}+\ve}}  = O(1)
\]
as long as $\ve>0$.\\

[Case D]  Recall that this case is only for $n\geq 4$.  In this case, we must have $k_{\max_3}^2 k_{\max_4} \gtrsim k_{\max}^2$.  Then condition~\eqref{eq:prop2} can be written as
\[
\f{|k|^s |ik|}{k_{\max_1}^{s-\ve} (k_{\max_2}k_{\max_3}k_{\max_4})^{\f{1}{2}+\ve}}\leq \f{|k|^{s+1}}{k_{\max}^{s+1+\ve}} = O(1),
\]
which is true for any $\ve>0$.
\end{proof}

\begin{remark}
Combining  Lemma~\ref{le:asym} and Lemma~\ref{le:lwp} with $\ve = s - \f{1}{2}$, we obtain that for any $v_1, \cdots ,v_n \in Y^s$ and $s>\f{1}{2}$ such that each $v_j(t)$ is supported in $t\in [-T,T]$,
\begin{equation}
\label{eq:asym}
\n{\rr^2}{Z^s}+ \n{\nr }{Z^s} \lesssim_\ve T^{\ve} \left(\prod_{j=1}^n \n{v_j}{Y^s} + \prod_{j=1}^m \n{v_j}{Y^s}  \right)
\end{equation}
where $\rr^2, \nr$ contain input $(v_1, \cdots, v_n)$ and $(v_1, \cdots, v_m)$.
\end{remark}

\section{Proof of Theorem~\ref{th:lwp}}
\label{Sec:3} 
Before we can prove Theorem~\ref{th:lwp}, we prove a weaker version of this theorem. 
\begin{proposition}\label{pro:lwp}
Let $s>\f{1}{2}$ and $f\in H^s$.  Then for some $T = T(\n{f}{H^{s}})>0$, the equation \eqref{eq:util} has a unique solution $u\in Y^s_T$ satisfying
\[
\n{u}{Y^s_T} \lesssim  \n{f}{H^{s}}.
\]
\end{proposition}

\begin{proof}
Let $\eta = \eta(t)$ be a smooth cutoff function with $\eta \equiv 1$ on $[-1,1]$ and $\eta \equiv 0$ for $|t|\geq 2$.  For $t\in [0,T]$, the equation \eqref{eq:util} can be formulated as 
\[
u (t) = e^{t\p_x^3} f + \int_0^t e^{(t-s)\p_x^3} \left(\rr^2 [ \eta(s/T)u(s)] +  \nr[\eta(s/T)u(s)]\right)\, ds=: \Ga_T[u].
\]
We will show that, for small $0<T\ll 1$, $\Ga_T$ is a contraction map in $Y^s_T$ inside a small ball~$B$ centered at $e^{t\p_x^3} f$ with radius $0<r < \n{f}{H^s}$ so that $\n{u}{Y_T^s} \leq C\n{f}{H^s}$ for all $u\in Y_T^s$.   Let $v \in Y^s$ with $v(t) = u(t)$ for $0\leq t \leq T$ and satisfying $\n{v}{Y^s} \leq 2 \n{u}{Y^s_T}$.  Then applying Proposition~\ref{pro:iteam} and Lemma~\ref{le:lwp} with $\ve = s-\f{1}{2}$,
\begin{align*}
\n{\Ga_T[u] - e^{t\p_x^3} f}{Y^s_T} &\lesssim \n{\int_0^t e^{(t-s)\p_x^3} \left(\rr^2 [ \eta(s/T)u(s)] +  \nr[\eta(s/T)u(s)]\right)\, ds}{Y_T^s}\\
	&\lesssim \n{\eta(t) \int_0^t e^{(t-s)\p_x^3} \left(\rr^2 [ \eta(s/T)v(s)] +  \nr[\eta(s/T)v(s)]\right)\, ds}{Y^s}\\
	&\lesssim \n{\rr^2 [ \eta(t/T)v(t)] +  \nr[\eta(t/T)v(t)]\, ds}{Z^s}\\
	&\lesssim T^\ve  \left(\n{v}{Y^{s}}^{n} + \n{v}{Y^{s}}^{m}\right)
	\lesssim T^\ve  \left(\n{u}{Y_T^{s}}^{n} + \n{u}{Y_T^{s}}^{m}\right)
	\lesssim \\
	&\lesssim T^\ve  \left(\n{f}{H^{s}}^{n} + \n{f}{H^{s}}^{m}\right).
\end{align*}
Select $T = T(\n{f}{H^s})$ so that the RHS above is smaller than the given radius $r$, we can show that $\Ga_T : B \to B$.

Next, we will show that $\Ga_T$ is a contraction on $B$ for a small $T>0$.  Let $u,v\in B$.   Then using analogous computations as the one directly above along with \eqref{eq:asym}, we can obtain 
\[
\n{\Ga_T[u] - \Ga_T[v]}{Y^s_T} \lesssim  T^\ve\left(\n{f}{H^{s}}^{n-1} + \n{f}{H^{s}}^{m-1} \right) \n{u - v}{Y^s_T} .
\]
Selecting a small $T = T(\n{f}{H^s})$ so that the implicit coefficient of $\n{u-v}{Y^s_T}$ is smaller than 1, we have proved our claim that $\Ga_T$ is a contraction map on $B$ for $T= T(\n{f}{H^s})$.  This proves an existence and uniqueness of solutions in $Y^s_T$ as well as an an estimate claimed in the statement.
\end{proof}
Now we are ready to prove Theorem~\ref{th:lwp}.  First, Let $v \in Y^s$ satisfying $v(t) = u(t)$ for all $0\leq t \leq T$ and  $\n{v}{Y^s} \leq 2 \n{u}{Y^s_T}$.  Let $u = \Ga_T[u]$ in $Y^s_T$ as given in the proof of Lemma above.  Then using Proposition~\ref{pro:iteam} and Lemma~\ref{le:lwp},
\begin{align*}
\n{u}{Y^s_T} &\leq \n{\eta(t) e^{t\p_x^3} f}{Y^s} + \n{\eta(t) \int_0^t e^{(t-s)\p_x^3} \left(\rr^2 [ \eta(s/T)v(s)] +  \nr[\eta(s/T)v(s)]\right)\, ds}{Y^s}\\
	& \lesssim \n{f}{H^s} + \n{\rr^2 [ \eta(t/T)v(t)] +  \nr[\eta(t/T)v(t)]}{Z^s}\\
	&\lesssim  \n{f}{H^s} + T^{\ve} \n{ u}{Y_T^s} \left(\n{ u}{Y_T^{\f{1}{2}+\ve}}^{n-1} + \n{ u}{Y_T^{\f{1}{2}+\ve}}^{m-1}\right).
\end{align*}
Taking $T \leq T(\n{f}{H^{\f{1}{2}+\ve}})$ given in Theorem~\ref{le:lwp}, we obtain
\[
\n{u}{Y^s_T} \lesssim \n{f}{H^s} + T^{\ve} \n{ u}{Y_T^s} \left(\n{f}{H^{\f{1}{2}+\ve}}^{n-1} + \n{f}{H^{\f{1}{2}+\ve}}^{m-1}\right)
\]
Select $T = T(\n{f}{H^{\f{1}{2}+\ve}})$ so that
\[
 T^{\ve} \left(\n{f}{H^{\f{1}{2}+\ve}}^{n-1} + \n{f}{H^{\f{1}{2}+\ve}}^{m-1}\right)\ll 1,
\]
we have the desired estimate.

Note that length of each time interval $T$ depends on $\n{f}{H^{\f{1}{2}+\ve}}$ only.  So, if $s\geq 1$ and we assume a priori control of $\n{u(t)}{H^1_x}$, then we can take a uniform time-step $T = T(\n{f}{H^1})$ and iterate this in time.  In this case, analogous estimate would state.
\[
\n{u}{Y^s_{[nT, (n+1)T]}} \lesssim \n{u(nT)}{H^s}.
\]
This proves Theorem~\ref{th:lwp}.
\section{Normal form transformation and asymmetric estimates}
\label{sec:4} 
Next we proceed with the required smoothing estimates which is achieved via normal form transformation.  Normal form method was introduced by Shatah \cite{Sh} in context of a scattering problem for cubic Klein-Gordon equation.  Since then, normal form and other related methods have been used for various nonlinear dispersive models to gain derivatives for nonlinearities.  See for instance \cite{titi, erdogan, fnls, KO, japanese, Oh, TT} and references therein.

Without normal form, we can only establish a non-smoothing estimate given in Lemma~\ref{le:lwp}.   Thus we further decompose our RHS and perform normal form transformation in order to obtain smoothing.

To that end,  we split $\nr$ into two components: $\nr[u] = \bb[u]+ \mathcal{HH}[u]$ where $\bb$ stands for a high-low frequency interaction and $\mathcal{HH}$ stands for a high-high frequency interaction.
\begin{align*}
\nr^1[u]_k &= \sum_{l=1}^n\sum_{\tiny\begin{array}{c} k_1 + \cdots + k_n = k\\ |k_l| \gg \max_{j\neq l} |k_j| \\ \sum_{j\neq l} k_j \neq 0\end{array}  } ia k \prod_{j=1}^n u_{k_j} + \sum_{l=1}^m \sum_{\tiny\begin{array}{c} k_1 + \cdots + k_m = k\\ |k_l| \gg \max_{j'\neq l} |k_{j'}|\\ \sum_{j'\neq l} k_{j'} \neq 0\end{array}  } ib k \prod_{j'=1}^m u_{k_{j'}}\\
&=  an\sum_{\tiny\begin{array}{c} k_1 + \cdots + k_n = k\\ |k_1| \gg \max_{j \geq 2} |k_j| \\ \sum_{j=2}^n k_j \neq 0\end{array}  } i k \prod_{j=1}^n u_{k_j} + bm\sum_{\tiny\begin{array}{c} k_1 + \cdots + k_m = k\\ |k_1| \gg \max_{j'\geq 2} |k_{j'}|\\ \sum_{j'=2}^m k_{j'} \neq 0\end{array}  } i k \prod_{j'=1}^m u_{k_{j'}}\\
&=: an \bb^n[u,\cdots, u]_k + bm\bb^m[u,\cdots,u]_k
\end{align*}
where we rearranged frequency indices to make the first input of $\bb^n$ and $\bb^m$ to carry the highest frequency component.  By construction, $\mathcal{HH}$ carries at least two high internal frequencies and is non-resonant.  Then we can rewrite \eqref{eq:util} as
\[
u_t + u_{xxx} = \rr^2[u] + an  \bb^n[u,\cdots, u] + bm \bb^m[u,\cdots,u] + \mathcal{HH}[u].
\]

In the next lemma, we will see that $\rr^2$ and $\mathcal{HH}$ are already smooth.  As for $\bb^n$ and $\bb^m$, we can see from Proposition~\ref{pro:disp} that these terms belong to Case A which is non-resonant.  For the part of $\bb^1$ which  has free solution $e^{t\p_x^3}f$ in the first component, we will filter out using normal form transform.   For this purpose, we define normal form operators~$T_n$ for any $n\geq 2$:
\[
T^n_{\mathcal{NF}} (f^1, \cdots, f^n) := \sum_{\tiny\begin{array}{c} k_1 + \cdots + k_n = k\\ k_1 \gg \max(k_2,\cdots, k_n)\\ k_2 + \cdots + k_n \neq 0\end{array} }\f{k}{H_n(k_1,\cdots,k_n)}  \prod_{j=1}^n f^j_{k_j} e^{i(k_1+\cdots+ k_n)x}, 
\]
where $\mathcal{NF}$ stands for the normal form symbol for the given $n$.  For any smooth functions $f = f(x)$ and $v = v(t,x)$, we have
\begin{equation}
\label{eq:remain}
(\p_t + \p_{x}^3) T^n_{\mathcal{NF}}  (e^{t\p_x^3} f,v) = \bb^n[ e^{t\p_x^3}f , v] + (n-1) T^n_{\mathcal{NF}}  (e^{t\p_x^3} f, v, (\p_t + \p_{x}^3) v).
\end{equation}
Then we define a new variable~$w$ 
\begin{equation}\label{eq:transf}
u = e^{t\p_x^3} f + a nT^n_{\mathcal{NF}} (e^{t\p_x^3} f,u, \cdots ,u) + bm T^m_{\mathcal{NF}} (e^{t\p_x^3} f,u,\cdots, u) + w
\end{equation}
where $u$ solves \eqref{eq:util}.  We will abbreviate the normal form terms as $T^n_{\mathcal{NF}}$ and $T^m_{\mathcal{NF}}$ when their inputs are the same as above.  Then $w$ satisfies the equation given by:
\begin{align}
w_t + w_{xxx} &= \rr^2[u] + \mathcal{HH}[u] \label{eq:w1}\\
	&+ an\bb^n[w, u,\cdots, u]  + bm \bb^m [w,u,\cdots, u] \label{eq:w2}\\
	& + (n-1) T^n_{\mathcal{NF}}  (e^{t\p_x^3} f, u,\cdots, u, \rr^2[u] + \nr[u]) \label{eq:w3}\\
	& + (m-1) T^m_{\mathcal{NF}}  (e^{t\p_x^3} f, u,\cdots, u, \rr^2[u] + \nr[u]) \label{eq:w4}\\
	& + a^2n^2\bb^n[ T^n_{\mathcal{NF}}  , u,\cdots, u] +b^2m^2\bb^m[T^m_{\mathcal{NF}} , u,\cdots, u]\label{eq:w5}\\
	& + abnm \left(\bb^n[T^m_{\mathcal{NF}}  , u,\cdots, u] +\bb^m[T^n_{\mathcal{NF}} , u,\cdots, u]\right) \label{eq:w6}\\
w\vert_{t=0} &= -T^n_{\mathcal{NF}} (f,\cdots, f) -T^m_{\mathcal{NF}}  (f, \cdots, f). \label{eq:w7}
\end{align}
Note that terms in \eqref{eq:w2}, \eqref{eq:w5} and \eqref{eq:w6} result from replacing the first input $u$ of $\bb^n$ and $\bb^m$ by \eqref{eq:transf}.  The term with free solution $e^{t\p_x^3} f$ as the first input is eliminated by the normal form.   Terms in \eqref{eq:w3} and \eqref{eq:w4} are the remainder terms from normal form resulting from \eqref{eq:remain}.  Finally, the initial data \eqref{eq:w7} can be obtained by using \eqref{eq:transf}.

The following lemma places the initial data \eqref{eq:w7} in $H^{s+1}$.  This also sets a ceiling for any possible smoothing to follow.
\begin{lemma}\label{le:t}
For any $n\geq 2$ and $s>\f{1}{2}$,
\[
\n{ T^n_{\mathcal{NF}}(u,v,\cdots, v)}{H^{s+1}}  \lesssim_\ve \n{u}{H^s} \n{v}{H^{\f{1}{2}+\ve}}^{n-1}.
\]
\end{lemma}
\begin{proof}
Under the frequency restriction in $T^n_{\mathcal{NF}}$, we must have that $H_n \gtrsim k_1^2 \sim k^2$ according to Proposition~\ref{pro:disp}.  So
\begin{align*}
\n{T^n_{\mathcal{NF}} (u,v, \cdots, v)}{H^{s+1}_0} &\lesssim \n{  \sum_{\tiny\begin{array}{c}k_1+\cdots k_n = k\\ k_1 \gg \max\{|k_2|,\cdots,|k_n|\}\end{array}}\f{k |k_1|^{s+1}}{H_n(k_1,\cdots, k_n)} u_{k_1} \prod_{j=2}^n v_{k_j}}{l^2_k}\\
&\lesssim \n{  \sum_{k_1+\cdots k_n = k}|k_1|^s u_{k_1} \prod_{j=1}^n v_{k_j}}{l^2_k}\\
	&\sim \n{|k_1|^s u_{k_1}}{l^2_{k_1}} \prod_{j=2}^n \n{v_{k_j}}{l^1_{k_j}}
	\lesssim_{\ve} \n{u}{H^s}\n{v}{H^{\f{1}{2}+\ve}}^{n-1}, 
\end{align*}
which proves the desired claim.
\end{proof}
Next we obtain a smoothing estimate for the first two terms on the RHS of \eqref{eq:w1}.
\begin{lemma} 
For $s>\f{1}{2}$  and $0<\ga <\min(1,2s-1)$, there exists an $\ve>0$ such that for any $0< T \ll 1$,
\[
\n{\rr^2[u]}{Z^{s+\ga}} + \n{\mathcal{HH}[u]}{Z^{s+\ga}} \lesssim_\ve T^{\ve} \n{u}{Y^s} \left(\n{u}{Y^{\min(s,1)}}^{n-1}+ \n{u}{Y^{\min(s,1)}}^{m-1}\right).
\]
\end{lemma}

\begin{proof}
This proof is a slight modification of the proof for Lemma~\ref{le:lwp} with $s_0 = s+\ga$, $s_1 = s$ and $s_2 = \min(s,1)$.  The symbol $\si$ for this estimate is again $\si = i k$ as in Lemma~\ref{le:lwp}.   As before, we will go through all cases of Proposition~\ref{pro:disp} and then apply appropriate conditions of Proposition~\ref{pro:main}.\\

[Case A]  Here, we have $H_n \gtrsim k_{\max}^2$.  Since both $\rr^2$ and $\mathcal{HH}$ has at least two internal frequencies comparable to the outside frequency~$k$, we can apply \eqref{eq:prop2} by observing
\[
\f{|k|^{s+\ga} |i k|}{|H_n|^{\f{1}{2}} |k_{\max_1}|^s |k_{\max_2}|^{\min(s,1)}} \lesssim   |k|^{\ga-\min(s,1)} = O(1)
\]
which is true as long as $\ga < \min(s,1)$  Since $\min(2s-1, 1) \leq \min(s,1)$ for any $s$, this condition is met from our assumption.\\

[Case B, C and D] Similar to the proof of Lemma~\ref{le:lwp}, the condition \eqref{eq:prop2} is written as 
\[
\f{ |k|^{s+\ga} |ik|}{k_{\max_1}^{s-\ve} \left(k_{\max_2} k_{\max_3} k_{\max_4}\right)^{\min(s,1)}} \lesssim |k|^{\ga +1  +\ve - 2 \min(s,1)}= O(1)
\] 
which holds as long as $\ga  < 2\min(s,1) - 1 = \min(2s-1,1)$.  This proves the desired result.
\end{proof}
For the remaining estimates, we will need an asymmetrical version of Proposition~\ref{pro:main}, which is given below.
\begin{proposition}
	\label{pro:main2}
Let $n\geq 2$ and $0< T\ll 1$.  Here, we only consider multi-linear operators $T^n_\si(u,v,\cdots, v)$ with has a frequency restriction $k_1 \sim k$.
\[
\n{T^n_\si (u,v, \cdots, v)}{Z^{s_0}_T} \lesssim_{\ve, n} T^{\ve} \n{u}{Y^{s_1}}\n{v}{Y^{s_2}}^{n-1}
\]
for some $s_0,s_1, s_2 > \f{1}{2}$ where $u := u(t) = \eta (t/T) u(t)$ for a smooth cut-off function $\eta$.  The inequality above is satisfied if either of the following conditions is met
\begin{itemize}
\item If dispersion weight $H_n$ is used, then we need
\begin{equation}\label{eq:prop1a}
\sup_{(k_1, \cdots, k_n) \in \Om_k} \f{ |k|^{s_0-s_1} |\si|(k_1,\cdots, k_n)}{ \lan{H_n}^{\f{1}{2}} k_{\max_2}^{s_2}}  = O(1)
\end{equation}
where $k_{\max_2} = \max\{|k_2|,\cdots, |k_n|\}$. \\

\item If dispersion weight $H_n$ is not used, then we need
\begin{equation}\label{eq:prop2a}
\sup_{(k_1, \cdots, k_n) \in \Om_k} \f{ |k|^{s_0-s_1+\ve} |\si|(k_1,\cdots, k_n)}{  \left( k_{\max_2} k_{\max_3} k_{\max_4}\right)^{s_{2}}}  = O(1).
\end{equation}
We use the convention that $k_{\max_j} = 1$ for $j>n$.
\end{itemize}

\end{proposition}
Proof of this proposition is identical to Proposition~\ref{pro:main}, except that rearrangement of indices occurs for $|k_2|\geq |k_3|\geq \cdots$.  We omit the details.
Now we are ready to we estimate the two terms in \eqref{eq:w2}.
\begin{lemma} 
For $s>\f{1}{2}$, $n\geq 2$  and any $\ga \in \br$, there exists an $\ve>0$ such that for any $0< T \ll 1$,
\[
\n{ \bb^n[w, u, \cdots, u]}{Z^{s+\ga}} \lesssim_\ve T^{\ve} \n{w}{Y^{s+\ga}} \left(\n{u}{Y^{\f{1}{2}+\ve}}^{n-1}+\n{u}{Y^{\f{1}{2}+\ve}}^{m-1}\right).
\]
\end{lemma}

\begin{proof}
We apply Proposition~\ref{pro:main2} with$\si = i k$, $s_0 = s_1 = s+\ga$ and $s_2 = \f{1}{2}+\ve$.  By construction of $\bb^n$, there cannot be frequency interactions of the form $B$ or $C$ of Proposition~\ref{pro:disp}, so we only need to consider Cases A and D.\\

[Case A]  Here $H_n \gtrsim k_1^2\sim k$, so condition~\eqref{eq:prop1a} is
\[
\f{|ik|}{|H_n|^{\f{1}{2}}} = O(1),
\]
which holds for any $s, \ga \in \br$.\\

[Case D]  This applies only for $n\geq 4$, where the condition~\eqref{eq:prop2a} becomes
\[
\f{|ik|}{ \left(k_{\max_2} k_{\max_3} k_{\max_4}\right)^{\f{1}{2}+\ve}} \lesssim |k|^{-\ve} = O(1)
\]
which holds for any $\ve >0$.
\end{proof}
Next, we estimate the normal form remainder terms in \eqref{eq:w3} and \eqref{eq:w4}.
\begin{lemma} \label{le:2n}
Let $n\geq 2$.  For $s>\f{1}{2}$ and $\ga <\min(s,1)$, there exists an $\ve>0$ such that for any $0< T \ll 1$,
\[
\n{T^n_{\mathcal{NF}} (e^{t\p_x^3} f, u,\cdots, u, \rr^2 [u]+ \nr[u]) }{Z^{s+\ga}} \lesssim_\ve T^{\ve} \n{f}{H^s} \left(\n{u}{Y^{\min (s,1)}}^{2n-2}+\n{u}{Y^{\min(s,1)}}^{n+m-2}\right).
\]
\end{lemma}

\begin{proof}
There are two estimates to consider here: (1) when $\rr^2$ or $\nr$ contains the $n$~th degree nonlinearity, (2) when $\rr^2$ or $\nr$ contains the $m$~th degree nonlinearity.  But essentially, we can deal with them together here, since we can set $m=n$ later if needed.  So we assume this estimate to be of the second kind mentioned.  Note that this is an $n+m-1$~multi-linear estimate.

Again, we use Proposition~\ref{pro:main2} with $s_0 = s+\ga$, $s_1 = s$, $s_2 = \f{1}{2}+\ve$ and the symbol 
\[
\si =\displaystyle  \f{ik\,(k_n +\cdots + k_{n+m-1}) }{H_n(k_1, \cdots, k_{n-1}, k_n +\cdots + k_{n+m-1}) }.
\]
Note that the denominator of this symbol $H_n$ should be distinguished by the expression $H_{n+m-1}$ of the condition~\ref{pro:main2}.  By construction, we know that $H_n \gtrsim k_1^2 \sim k$ as well as $|k_n +\cdots + k_{n+m-1}| \ll |k_1|$.  Thus we have $\si = O(1)$. Now we go through the cases:\\

[Case A] Here, we have $H_{n+m-1}\gtrsim k_{\max}^2$ and $\si = O(1)$.  So using the condition \eqref{eq:prop1a}
\[
\f{|k|^{\ga} |\si| }{|H_{n+m-1}|^{\f{1}{2}}}\lesssim |k|^{\ga -1} = O(1)
\]
which holds as long as $\ga \leq 1$.\\

[Case B] Recall that the operator $T_{\mathcal{NF}}^n$ comes with the frequency-restriction:
\[
|k_1| \gg \max (|k_2|, \cdots, |k_{n-1}|, |k_n +\cdots + k_{n+m-1}|).
\]
Then, in order to be in Case B, we must have either $k_1 = k$ or $k_{j_0} = k$ for $j_0 \in\{ n, \cdots, n+m-1\}$.  In the former scenario, we have
\[
k_2 + \cdots + k_{n+m-1} = 0 \implies \max \{|k_2|, \cdots, |k_{n-1}|\} \gtrsim  |k_n +\cdots + k_{n+m-1}|.
\]
Without loss of generality, say $k_2\gtrsim k_n +\cdots k_{n+m-1}$.  Then since $H_n$ in the denominator gives a gain of $k_1^2$, we use the condition \eqref{eq:prop2a} to obtain
\[
\f{|k|^{\ga+\ve} |\si|}{|k_2|^{\min(s,1)}} = \f{|k|^{1+\ga+\ve} \,|k_n +\cdots, + k_{n+m-1}| }{H_n |k_2|^{\min(s,1)}} \lesssim |k_1|^{\ga-1 +\ve} k_2^{1 -\min (s,1)}. 
\]
For $\ga < 1$, this is controlled by $|k_1|^{\ga - \min (s,1)}$ which is uniformly bounded if $\ga < \min (s,1)$.

In the latter scenario, say without loss of generality that $k_{n+m-1} = k$. Then,
\[
k_1 + \cdots + k_{n+m-2} = 0 \implies \max\{ |k_2|, \cdots, |k_{n+m-2}|\} \gtrsim |k_1| \sim |k|.
\]
Without loss of generality, let $k_2 \gtrsim k_1 \sim k$.   Then using the criterion \eqref{eq:prop2a}, 
\[
\f{|k|^{\ga+\ve} |\si|}{(|k_2| |k_{n+m-1}|)^{\min(s,1)}} \lesssim |k|^{\ga + \ve - 2\min(s,1)}     
\]
which is uniformly bounded if $\ga <\min(2s,2)$.\\

[Case C $\&$ D] In these cases,  \eqref{eq:prop2a} gives
\[
\f{|k|^{\ga+\ve} |\si|}{ (k_{\max_2} k_{\max_3}k_{\max_4})^{\min(s,1)}} \lesssim k_{\max}^{\ga+\ve -  2\min(s,1)}
\]
which leads to the same estimate as before.
Applying Proposition~\ref{pro:main} in each case, we obtain the desired estimate.
\end{proof}
Next, we deal with the two terms in equation of \eqref{eq:w5}.

\begin{lemma} \label{le:mix}
Let $n\geq 2$.  For $s>\f{1}{2}$  and $0<\ga <\min(1,3s-1)$, there exists an $\ve>0$ such that for any $0< T \ll 1$,
\[
\n{ \bb^n[T^n_{\mathcal{NF}} (e^{t\p_x^3} f, u,\cdots, u), u,\cdots, u]}{Z^{s+\ga}} \lesssim_\ve T^{\ve} \n{f}{H^s} \n{u}{Y^{\min(s,1)}}^{2n-2}.
\]
\end{lemma}

By a normal size-estimate, this estimate should fail.  But there is a cancellation structure here that can be utilized.  This cancellation structure has an explicit algebraic expression in case of mKdV ($n=3$) and was used to establish a well-posedness of periodic mKdV below $s=\f{1}{2}$ in  \cite{japanese}.  It is remarkable that analogous cancellation structure exists for gKdV even with a non-explicit algebraic expression.

\begin{proof}
Again, this is a $2n-1$~multi-linear estimate.  Symbol $\si$ can be put in the following form:
\[
\si = \f{ik (k_1 + \cdots+ k_n)}{H_n (k_1, \cdots, k_n)}.
\]
What makes this symbol worse that the similar-looking symbol from the previous lemma is that two derivatives in the numerator are both high-frequency.  So the trick used in case B of this estimate in Lemma~\ref{le:2n} does not work here.  In fact, condition~\ref{eq:prop2} fails for any $\ga >0$ in Case B of Proposition~\ref{pro:disp}.  Before we perform a size-estimate, we will first harvest cancellations from the worst term here.

Note that due to the frequency restriction of $\bb^n$ and $T^n_{\mathcal{NF}} $, we must have $k_1 \gg \max\{|k_2|, \cdots, |k_{2n-1}|\}$.  So the only possible cases of Proposition~\ref{pro:disp} are Case A and Case B.  Furthermore, Case B can occur when $k_2 + \cdots + k_{2n-1} = 0$ which means $k_1 = k$.  Consider this term with the frequency restriction given by:
\begin{equation}\label{eq:rest}
k_2 + \cdots + k_{2n-1} = 0, \qquad \max_{j\geq 2} |k_j| \ll k.
\end{equation}
We decompose this nonlinearity as follows:
\[
T^{2n-1}_\si = T_\si^B + T_\si^{A}
\]
where $T_\si^B$ is the $2n-1$~multi-linear Fourier multiplier with symbol $\si$ and frequency restrictions corresponding to \eqref{eq:rest}.  For $T_\si^{A}$, there will be enough smoothing to perform a direct size estimate.  For $T_{\si}^B$, we must first further decompose the symbol to take advantage of cancellation.

Note that we are guaranteed of $k_2 + \cdots + k_n \neq 0$ by construction of $T$.  Define
\[
\mu =  \f{i}{3(k_2 + \cdots + k_n)}.
\]
We will show that $T_{\si-\mu}^B$ is bounded using Proposition~\ref{pro:main2}, while $T_{\mu}^B$ cancels out entirely.  First, we estimate the size of $\si-\mu$ here.  Argument is made separately for $n=2$ (quadratic normal form), $n=3$ (cubic normal form) and $n\geq 4$.\\

\textbf{$\si -\mu$ estimate for $n=2$:}  This is the simplest case where $T_{\si-\mu}^B \equiv 0$.  Since $H_2 = 3k_1 k_2(k_1+k_2)\neq 0$, so we have
\[
\si - \mu = \f{ik}{3k_1 k_2} - \f{i}{3k_2} = \f{  i (k_2+ k_3) }{3k_1 k_2}  = 0
\]
since $k_2 + k_3 = 0$ by \eqref{eq:rest}.\\

\textbf{$\si -\mu$ estimate for $n=3$:}  Recall that $H_3 = 3 (k_1 + k_2) (k_2+ k_3)(k_3+k_1)\neq 0$.  Then
\begin{align*}
\si - \mu &= \f{ik(k_1 + k_2+k_3)}{3 (k_1+k_2)(k_2+ k_3)(k_3+k_1)} - \f{i}{3(k_2 + k_3)} \\
	&= \f{  i k_1 (k_1+k_2+k_3) - i (k_3+k_1)(k_1+k_2) }{3(k_1+k_2)(k_2 + k_3)(k_3+k_1)}\\
	& = \f{   - i k_2 k_3 }{3(k_1+k_2)(k_2 + k_3)(k_3+k_1)}\\
	&\lesssim \f{k_{\max_2} k_{\max_3}}{k_{1}^2}.
\end{align*}

\textbf{$\si -\mu$ estimate for $n\geq 4$:}  For this argument, we will use the notation $\wt{k_j} := k_j + \cdots + k_n$ similar to the one introduced during the proof of Proposition~\ref{pro:disp}.  However, note that the summation here stops at index $n$, where there are a total of $2n-1$ indices.
\[
\si - \mu =   \f{ik \wt{k_1} }{H_n} -\f{i}{3\wt{k_2}} = \f{  3i k_1 \wt{k_1}  \wt{k_2} - iH_n }{3\wt{k_2} H_n}.
\]
Now, recall from the \eqref{eq:hn} that $H_n$ is equal to 
\[
H_n = 3 k_1 \wt{k_1} \wt{k_2} + 3 k_{2} \wt{k_2} \wt{k_3} + \cdots + 3 k_{n-1} \wt{k_{n-1}} \wt{k_{n}},\quad .
\]
Combining this with above, we obtain
\[
\si - \mu= - 3i\f{k_{2} \wt{k_2} \wt{k_3} + \cdots +  k_{n-1} \wt{k_{n-1}} \wt{k_{n}}}{\wt{k_2} H_n} = -3i\f{k_2 \wt{k_3}}{H_n} -3i \f{\sum_{j=3}^n k_{j-1} \wt{k_{j-1}} \wt{k_j}}{\wt{k_2} H_n}.
\]
The numerator of the first fraction on RHS is bounded by $|k_{\max_2}| |k_{\max_3}|$,  which can be controlled by the denominator.  So the first fraction has size $|k_{1}|^{-1}$.  For the second fraction, we examine the numerator:
\begin{align*}
k_{j-1} \wt{k_{j-1}} \wt{k_j} &= k_{j-1} (\wt{k_2} - k_2 - \cdots - k_{j-2})   (k_j + \cdots + k_n)\\
&= \wt{k_2} k_{j-1} \wt{k_j} - \sum_{l < j-1} \sum_{m>j-1}  k_l k_{j-1}  k_m.   
\end{align*}
It is important to notice in this computation is that the first term contains $\wt{k_2}$,  which will cancel with the denominator, and the remaining term is a product of three non-overlapping frequencies.  Thus, we can write
\[
\si - \mu =  -3i\f{k_2 \wt{k_3}}{H_n} -3i \f{\sum_{j=3}^n k_{j-1}  \wt{k_j}}{H_n} - 3i\f{\sum_{j=3}^n \sum_{l<j-1} \sum_{m>j-1} k_l k_{j-1} k_m}{\wt{k_2} H_n}.
\]
Since $H_n \gtrsim k_1^2$, we have
\begin{equation}\label{eq:simu}
|\si - \mu| \lesssim \f{k_{\max_2} k_{\max_3} k_{\max_4} }{k_1^2}.
\end{equation}
We will use the above estimate for all $n\geq 3$.  Since this is a $2n-1$~multi-linear estimate, $k_{\max_4}$ still makes sense for $n=3$.   For $n=2$, we observed that the entire term $T_{\si-\mu}^B$ cancels out.  Now, we have
\[
 T^{2n-1}_\si = T_{\si}^{A} + T_{\si - \mu}^B + T_{\mu}^B.
\]
We will use Proposition~\ref{pro:main2} to estimate the first two terms.  Also, we will show that the last term cancels completely.\\

\textbf{Estimate for $T_{\si}^{A}$:}   By construction $H_n \gtrsim k_1^2$, and also since we are in Case A, $H_{2n-1} \gtrsim k_1^2$.  To satisfy \eqref{eq:prop1a}, 
\[
\f{|k|^{\ga} |ik| |k_1 + \cdots + k_n|}{|H_n| |H_{2n-1}|^{\f{1}{2}}}\lesssim |k_1|^{\ga-1} = O(1) 
\]
we need $\ga \leq 1$. This concludes the estimate for $T_{\si}^A$.\\

\textbf{Estimate $T_{\si - \mu}^B$:} We use \eqref{eq:prop2a} and \eqref{eq:simu} to obtain this estimate.  We need to control
\begin{align*}
\f{|k|^{\ga+\ve} |\si - \mu|}{  (k_{\max_2} k_{\max_3} k_{\max_4})^{\min(s,1)}} &\lesssim \f{|k|^{\ga+\ve}k_{\max_2} k_{\max_3} k_{\max_4}}{ k^2 ( k_{\max_2} k_{\max_3} k_{\max_4})^{\min(s,1)}} \\
&\lesssim |k_1|^{\ga - 2 + \ve} (k_{\max_2} k_{\max_3} k_{\max_4})^{1-\min(s,1)}\\
&\lesssim |k_1|^{\ga + 1 + \ve -3\min(s,1) }.
\end{align*}
Here we need $\ga < 3 \min (s,1) - 1 = \min (3s-1, 2)$.\\
 
\textbf{Cancellation of $T_{\mu}^B$:}  The $k$~th Fourier coefficient of $T_{\mu}^B$ is written as 
\[
e^{itk^3 }f_k e^{ik x} \sum_{\tiny\begin{array}{c} k_2 + \cdots + k_{2n-1} = 0\\ k \gg \max(k_2,\cdots, k_{2n-1})\\ k_2 + \cdots + k_n \neq 0 \end{array} }\f{i}{3(k_2 + \cdots + k_n)}  \prod_{j=2}^{2n-1} u_{k_j} .
\]
The sum above can be written as
\[
 \sum_{l\in \bz^* }\f{i}{3l} \sum_{\tiny\begin{array}{c} k \gg \max(k_2,\cdots, k_{n})\\ k_2 + \cdots + k_n = l \end{array} } \prod_{j=2}^{n} u_{k_j}  
 \sum_{\tiny\begin{array}{c} k \gg \max(k_{n+1},\cdots, k_{2n-1})\\ k_{n+1} + \cdots + k_{2n-1} = -l \end{array} } \prod_{j'=n+1}^{2n-1} u_{k_{j'}} .
 \]
Due to the symmetry in frequency indices $(k_2, \cdots, k_n)$ and $(k_{n+1}, k_{2n-1})$, this sum cancels out completely.  For instance, we can split the sum into $l>0$ and $l<0$ and observe that they cancel each other out.  Roughly speaking, this is equivalent to the calculation $\int_{\bt} (\p_x^{-1} g) g dx=0$, for appropriately chosen function $g: \hat{g}(0)=0$. Thus,  $T_{\mu}^B \equiv 0$ which completes the proof. 
\end{proof}
Finally, the next lemma deals with the mixed terms in \eqref{eq:w6}.
\begin{lemma} 
Let $n\neq m\geq 2$.  For $s>\f{1}{2}$  and $0<\ga <\min(1,3s-1)$, there exists an $\ve>0$ such that for any $0< T \ll 1$,
\[
\n{\bb^n[T^m_{\mathcal{NF}}  , u,\cdots, u]}{Z^{s+\ga}}+\n{\bb^m[T^n_{\mathcal{NF}} , u,\cdots, u]}{Z^{s+\ga}} \lesssim_\ve T^{\ve} \n{f}{H^s} \n{u}{Y^{\min(s,1)}}^{n+m-2}.
\]
\end{lemma}

\begin{proof}
Here both of them are $n+m-1$~multi-linear estimates.  Two symbols interact with each other to result in a surprising cancellation.  Symbol involved in this estimate is 
\[
\si =  \f{ik (k_1 + \cdots+ k_m)}{H_m (k_1, \cdots, k_m)} + \f{ik (k_1 + \cdots+ k_n)}{H_n (k_1, \cdots, k_n)}  .
\]
Once again, frequency restriction forces $k_1 \gg \max( |k_2| ,\cdots, |k_{n+m-1}|)$, so the only cases to consider are Case A and B of Proposition~\ref{pro:disp}.  As before, we decompose
\[
T^{2n-1}_\si = T_{\si}^A + T_{\si-\mu}^B + T_{\mu}^B
\]
where the restriction $B$ is defined analogously as \eqref{eq:rest} and $\mu$ is given by
\[
\mu  = \f{i}{3(k_2 + \cdots + k_m)} + \f{i}{3(k_2+ \cdots + k_n)}.
\]
Following the algebra in the previous proof leads to the size estimate for $\si - \mu$ given in \eqref{eq:simu}.  Thus, estimates for the first two terms are identical to the one given in the proof of Lemma~\ref{le:mix}.

It remains to show the cancellation for $T_{\mu}^B$.
 Its $k$~th Fourier coefficient is written as 
\[
e^{itk^3 }f_k e^{ik x} \sum_{\tiny\begin{array}{c} k_2 + \cdots + k_{2n-1} = 0\\ k \gg \max(k_2,\cdots, k_{2n-1})\\ k_2 + \cdots + k_n \neq 0 \end{array} }\left(\f{i}{3(k_2 + \cdots + k_n)}+ \f{i}{3(k_2 + \cdots + k_m)}\right)  \prod_{j=2}^{n+m-1} u_{k_j} .
\]
The sum above can be written as
\begin{align*}
& \sum_{l\in \bz^* }\f{i}{3l} \sum_{\tiny\begin{array}{c} k \gg \max(k_2,\cdots, k_{n})\\ k_2 + \cdots + k_n = l \end{array} } \prod_{j=2}^{n} u_{k_j}  
 \sum_{\tiny\begin{array}{c} k \gg \max(k_{n+1},\cdots, k_{n+m-1})\\ k_{n+1} + \cdots + k_{n+m-1} = -l \end{array} } \prod_{j'=n+1}^{n+m-1} u_{k_{j'}} \\
 &+ \sum_{l\in \bz^* }\f{i}{3l} \sum_{\tiny\begin{array}{c} k \gg \max(k_2,\cdots, k_{m})\\ k_2 + \cdots + k_m = l \end{array} } \prod_{j=2}^{m} u_{k_j}  
 \sum_{\tiny\begin{array}{c} k \gg \max(k_{m+1},\cdots, k_{n+m-1})\\ k_{m+1} + \cdots + k_{n+m-1} = -l \end{array} } \prod_{j'=m+1}^{n+m-1} u_{k_{j'}} 
 \end{align*}
Note that we can use symmetry in the second sum to reassign $(k_{m+1},\cdots, k_{n+m-1})\mapsto (k_1,\cdots, k_n)$ and produce
\[
-\sum_{l\in \bz^* }\f{i}{3l} \sum_{\tiny\begin{array}{c} k \gg \max(k_2,\cdots, k_{n})\\ k_2 + \cdots + k_n = l \end{array} } \prod_{j=2}^{n} u_{k_j}  
 \sum_{\tiny\begin{array}{c} k \gg \max(k_{n+1},\cdots, k_{n+m-1})\\ k_{n+1} + \cdots + k_{n+m-1} = -l \end{array} } \prod_{j'=n+1}^{n+m-1} u_{k_{j'}} 
\]
 which cancels completely with the first sum.  This proves our desired estimate.
\end{proof}

\section{Proof of Theorem~\ref{th:smoothing}}
\label{sec:5}

From the equation for $w$, note that we have
\[
\n{w}{Y^{s+\ga}_T} \lesssim \n{T^n_{\mathcal{NF}}(f,\cdots f)}{H^{s+\ga}} +  \n{T^m_{\mathcal{NF}}(f,\cdots f)}{H^{s+\ga}}  + \n{\eqref{eq:w1}}{Z^{s+\ga}} + \cdots \n{\eqref{eq:w6}}{Z^{s+\ga}}
\]

Using Lemmas from this section, we can bound these by
\begin{align*}
\n{w}{Y^{s+\ga}_T} &\lesssim \n{f}{H^s} (\n{f}{H^{\f{1}{2}+\ve}}^n + \n{f}{H^{\f{1}{2}+\ve}}^m) \\
	&+ T^{\ve} \n{u}{Y_T^s} \left(\n{u}{Y_T^{\min(s,1)}}^{n-1}+ \n{u}{Y_T^{\min(s,1)}}^{m-1}\right)\\
	&+T^{\ve} \n{w}{Y_T^{s+\ga}} \left(\n{u}{Y_T^{\f{1}{2}+\ve}}^{n-1}+\n{u}{Y_T^{\f{1}{2}+\ve}}^{m-1}\right)\\
	&+T^{\ve} \n{f}{H^s} \left(\n{u}{Y_T^{\min (s,1)}}^{2n-2}+\n{u}{Y_T^{\min(s,1)}}^{n+m-2}\right)\\
	&+ T^{\ve} \n{f}{H^s} \left(\n{u}{Y_T^{\min(s,1)}}^{2n-2}  \n{u}{Y_T^{\min(s,1)}}^{n+m-2}\right).
\end{align*}
Now, selecting $T< T(\n{f}{H^{\f{1}{2}+\ve}})$ from Theorem~\ref{th:lwp}, we get
\begin{align*}
\n{w}{Y_T^{s+\ga}} &\lesssim \n{f}{H^s} (\n{f}{H^{\f{1}{2}+\ve}}^n + \n{f}{H^{\f{1}{2}+\ve}}^m) \\
	&+ T^{\ve} \n{f}{H^s} \left(\n{f}{H^{\min(s,1)}}^{n-1}+ \n{f}{H^{\min(s,1)}}^{m-1}\right)\\
	&+T^{\ve} \n{w}{Y_T^{s+\ga}} \left(\n{f}{H^{\f{1}{2}+\ve}}^{n-1}+\n{f}{H^{\f{1}{2}+\ve}}^{m-1}\right)\\
	&+T^{\ve} \n{f}{H^s} \left(\n{f}{H^{\min (s,1)}}^{2n-2}+\n{f}{H^{\min(s,1)}}^{n+m-2}\right)\\
	&+ T^{\ve} \n{f}{H^s} \left(\n{f}{H^{\min(s,1)}}^{2n-2} +  \n{f}{H^{\min(s,1)}}^{n+m-2}\right).
\end{align*}
Selecting $T(\n{f}{H^{\f{1}{2}+\ve}})$ small so that 
\[
T^\ve  \left(\n{f}{H^{\f{1}{2}+\ve}}^{n-1}+\n{f}{H^{\f{1}{2}+\ve}}^{m-1}\right) \ll 1,
\]
we obtain that $\n{w}{Y_T^s} \leq C(\n{f}{H^{\min(s,1)}}) \n{f}{H^s}$.\\

Recalling \eqref{eq:transf}, for $\ga < 1$,
\begin{align*}
\n{u(t) - e^{t\p_x^3} f}{C_t^0([0,T]; H_x^{s+\ga})}  &\lesssim \n{T_{\mathcal{NF}}^n}{H^{s+\ga}} + \n{T_{\mathcal{NF}}^m}{H^{s+\ga}} +\n{w(t)}{H^{s+\ga}}\\
	&\lesssim \n{u(t)}{H^s} \left(\n{u(t)}{H^{\f{1}{2}+\ve}}^{n-1}+ \n{u(t)}{H^{\f{1}{2}+\ve}}^{m-1}\right) + \n{w(t)}{H^{s+\ga}}\\
	&\lesssim \n{u}{Y_T^s} \left(\n{u}{Y_T^{\f{1}{2}+\ve}}^{n-1}+ \n{u}{Y_T^{\f{1}{2}+\ve}}^{m-1}\right) + \n{w}{Y_T^{s+\ga}}\\
		&\lesssim_{\n{f}{H^{\min(s,1)}}} \n{f}{H^s}.
\end{align*}
This proves Theorem~\ref{th:smoothing}.

\begin{remark}
Note that, for $s\geq 1$, if we assume a priori control of $H^1$~norm, the length of time increment is uniform over a global iteration.  Also, the implicit constant depends on the $H^1$ norm, which is assumed to be  controlled.  So for $s\geq 1$ and $\ve>0$, 
\begin{equation}\label{eq:smooth}
\n{u(t) - e^{(t-nT)\p_x^3} u(nT)}{C^0_t( [nT, (n+1)T], H_x^{s+1-\ve})} \leq C(s,\ve, P,\n{f}{H^1}) \n{u(nT)}{H^s}.
\end{equation}
\end{remark}

\section{Proof of Theorem~\ref{th:main}}
\label{sec:6} 
First we will prove the statement for $s\in (1,2)$.  Take $T= T(\n{f}{H^1})$ by selecting $\ve = \f{1}{2}$ in Theorem~\ref{th:smoothing}.  We will use the statement of this Theorem in form of \eqref{eq:smooth}.\\

Let $t\in [nT, (n+1)T)$.  Define Fourier projection operators $\bp_{\leq n}$ and $\bp_{> n}$ to be Fourier frequency restrictions to $|k|\leq n$ and $|k|>n$ respectively.  Then using the decomposition
\[
u(t) = e^{(t-nT)\p_x^3} u(nT) + v_n(t),
\]
we have
\begin{equation}\label{eq:high}
\n{\bp_{>n} u(t)}{H^s} \leq \n{\bp_{>n} u(nt_0)}{H^s} + \n{\bp_{>n} v_n(t)}{H^s} 
\end{equation}
Take the second term from RHS of \eqref{eq:high}.  Using \eqref{eq:smooth} and a priori control of $H^1$ norm, we get
\[
\n{\bp_{>n} v_n(t)}{H^s} = \n{D^{s-2+\ve}\bp_{>n}  D^{2-\ve} v_n(t)}{L^2} \lesssim n^{s-2+\ve} \n{v_n(t)}{H^{2-\ve}} \lesssim n^{s-2+\ve} C(\n{f}{H^1}).
\]
Now, take the first term from RHS of \eqref{eq:high}.  We can iterate this term by
 \[
 u(nT)= e^{T\p_x^3} u((n-1)T) + v_{n-1}(nT).
 \]
Then, we have 
 \[
 \n{\bp_{>n} u(nT)}{H^s} \leq  \n{\bp_{>n-1} u(nT)}{H^s} \leq \n{\bp_{>n-1} u((n-1)T)}{H^s} +\n{\bp_{>n-1} v_{n-1} (nT)}{H^s} 
\]
where the last term on the RHS above is bounded by $(n-1)^{s-2+\ve} C(\n{f}{H^1})$ by applying \eqref{eq:smooth}.  Continuing this iteration all the way to the interval $[0,T]$ yield the following estimate.
\[
\n{\bp_{>n} u(t)}{H^s} \leq \n{f}{H^s} +C(\n{f}{H^1})  \sum_{k=1}^n k^{s-2+\ve}\lesssim  \n{f}{H^s} +C(\n{f}{H^1}) n^{s-1+\ve}.
\]
For the low-frequency component, a trivial estimate below is sufficient:
\[ 
\n{\bp_{\leq n} u(t)}{H^s} \lesssim n^{s-1} \n{u(t)}{H^1} = n^{s-1} C(\n{f}{H^1}).
\]
Since $n \sim \lan{t}$ (implicit constant depending on $T$), this gives 
\[
\n{u(t)}{H^s} \leq \n{\bp_{> n} u(t)}{H^s} +\n{\bp_{\leq n} u(t)}{H^s} \lesssim_{\n{f}{H^s}} n^{s-1+\ve} \sim_T \lan{t}^{s-1+\ve}.
\]
This proves the desired statement for $s\in [1,2)$.\\

We will use induction for $s>2$.  For some integer $N\geq 2$, let the claim hold for $s_0\in [N-1, N)$.  Then we will show that it holds for $s\in [N,N+1)$.  Note that Theorem~\ref{th:smoothing} allows for a uniform time-step $T= T(\n{f}{H^1})$.  Given a time $t \in [nT, (n+1)T)$, again we write
\[
u(t) = e^{(t-nT)\p_x^3} u(nT) + v_n(t).
\]
As before, application of \eqref{eq:smooth} yields
\begin{align*}
\n{\bp_{>n} v_n(t)}{H^{s}}& = \n{D^{s-N-1+2\ve}\bp_{>n}  D^{N+1-2\ve} v_n(t)}{L^2}\\
 &\lesssim n^{s-N-1+2\ve} \n{v_n(t)}{H^{N+1-2\ve}} \\
 &\lesssim_{\n{f}{H^1}} n^{s-N-1+2\ve} \n{u(nT)}{H^{N-\ve}}.
\end{align*}
Using the induction hypothesis, note  
\[
\n{u(nT)}{H^{N-\ve}} = O_{\n{f}{H^1}} (\lan{nT}^{N-1}) \lesssim_{T, \n{f}{H^1}} n^{N-1}.
\]  
Then
\[
\n{\bp_{>n} v_n(t)}{H^{s}}\lesssim_{\n{f}{H^1}} n^{s-2+2\ve}
\]

Now we can iterate the same way as before to arrive $\n{\bp_{>n} u(t)}{H^{s}} \lesssim n^{s-1+3\ve}$ after summation. The low-frequency estimate $\n{\bp_{\leq n} u(t)}{H^{s}} \lesssim n^{s-1} \n{f}{H^1}$ is still identical.  So together we get
\[
\n{u(t)}{H^s} \lesssim_{\n{f}{H^1}} n^{s-1+3\ve} \lesssim_{T} \lan{t}^{s-1+3\ve}.
\]
This proves Theorem~\ref{th:main} for all $s>1$.\\

The following statement immediately follows from analogous computations as above.  Its proof is omitted.
\begin{corollary}
For any $k\in\bn$, if $H^k$ norm of solutions for \eqref{eq:gKdV} can be controlled, then the global-in-time solution $u$ with initial value $f\in H^s(\bt)$ for $s > k$ satisfies the following polynomial-in-time bound:
\[
\n{u(t)}{H^{s}} \lesssim_{\ve, P, \n{f}{H^s}} \lan{t}^{s-k+ \ve} \qquad \textnormal{ for any } \ve >0.
\]
\end{corollary}

\end{document}